\documentclass[12pt,reqno]{amsart}
\usepackage{}
\usepackage{amsfonts}
\usepackage{a4wide}
\allowdisplaybreaks
\numberwithin{equation}{section}

\newtheorem{theorem}{Theorem}[section]
\newtheorem{proposition}[theorem]{Proposition}
\newtheorem{corollary}[theorem]{Corollary}
\newtheorem{lemma}[theorem]{Lemma}

\theoremstyle{definition}

\newtheorem{remark}[theorem]{Remark}

\newcommand{\R}{\mathbb{R}}
\newcommand{\N}{\mathbb{N}}

\begin{document}

\title
 [Double critical problem involving fractional Laplacian with Hardy term ]
 {The existence of a nontrivial weak solution to a double critical problem involving fractional Laplacian in ${\R}^n$ with a Hardy term}\footnotetext{*This work was supported by Natural Science Foundation of China (Grant No. 11771166), Hubei Key Laboratory of Mathematical Sciences and Program for Changjiang Scholars and Innovative Research Team in University $\# $ IRT17R46.}

\maketitle
 \begin{center}
\author{Gongbao Li}
\footnote{Corresponding Author: Gongbao Li. G. Li's Email addresses: ligb@mail.ccnu.edu.cn; T. Yang's Email addresses: yangt@mails.ccnu.edu.cn.}
and
\author{Tao Yang}
\end{center}

\begin{center}
\address{Hubei Key Laboratory of Mathematical Sciences and School of Mathematics and Statistics, Central China
Normal University, Wuhan, 430079, P. R. China }
\end{center}

\maketitle
\begin{abstract}
In this paper, we consider the existence of nontrivial weak solutions to a double critical problem involving fractional Laplacian with a Hardy term:
\begin{equation} \label{eq0.1}
(-\Delta)^{s}u-{\gamma} {\frac{u}{|x|^{2s}}}= {\frac{{|u|}^{ {2^{*}_{s}}(\beta)-2}u}{|x|^{\beta}}}+ \big [   I_{\mu}* F_{\alpha}(\cdot,u)  \big](x)f_{\alpha}(x,u),
  \ \ u \in {\dot{H}}^s(\R^{n})
\end{equation}
where $s \in(0,1)$, $0\leq \alpha,\beta<2s<n$, $\mu \in (0,n)$, $\gamma<\gamma_{H}$, $I_{\mu}(x)=|x|^{-\mu}$, $F_{\alpha}(x,u)=\frac{      {|u(x)|}^{ {2^{\#}_{\mu} }(\alpha)}      }{  {|x|}^{  {\delta_{\mu} (\alpha)} }  }$, $f_{\alpha}(x,u)=\frac{      {|u(x)|}^{{ 2^{\#}_{\mu} }(\alpha)-2}u(x)      }{  {|x|}^{  {\delta_{\mu} (\alpha)}            }  }$, $2^{\#}_{\mu} (\alpha)=(1-\frac{\mu}{2n})\cdot 2^{*}_{s} (\alpha)$,  $\delta_{\mu} (\alpha)=(1-\frac{\mu}{2n})\alpha$,   ${2^{*}_{s}}(\alpha)=\frac{2(n-\alpha)}{n-2s}$ and  $\gamma_{H}=4^s\frac{\Gamma^2(\frac{n+2s}{4})}
{\Gamma^2(\frac{n-2s}{4})}$. We show that problem (\ref{eq0.1}) admits at least a weak solution under some conditions.

To prove the main result, we develop some useful tools based on a weighted Morrey space. To be precise, we discover the embeddings
\begin{equation} \label{eq0.2}
{\dot{H}}^s(\R^{n}) \hookrightarrow  {L}^{2^*_{s}(\alpha)}(\R^{n},|y|^{-\alpha}) \hookrightarrow L^{p,\frac{n-2s}{2}p+pr}(\R^{n},|y|^{-pr})
\end{equation}
where $s \in (0,1)$, $0<\alpha<2s<n$, $p\in[1,2^*_{s}(\alpha))$, $r=\frac{\alpha}{ 2^*_{s}(\alpha) }$; We also establish an improved Sobolev inequality.

By using mountain pass lemma along with an improved Sobolev inequality, we obtain a nontrivial weak solution to problem (\ref{eq0.1}) in a direct way.  It is worth while to point out that the improved Sobolev inequality could be applied to simplify the proof of the existence results in \cite{NGSS} and \cite{RFPP}.


\end{abstract}

\maketitle

\section{Introduction and Main Result}

\setcounter{equation}{0}
In this paper, we consider the existence of nontrivial weak solutions to a double critical problem involving fractional Laplacian with a Hardy term:
\begin{equation} \label{eq1.1}
(-\Delta)^{s}u-{\gamma} {\frac{u}{|x|^{2s}}}= {\frac{{|u|}^{ {2^{*}_{s}}(\beta)-2}u}{|x|^{\beta}}}+ \big [   I_{\mu}* F_{\alpha}(\cdot,u)  \big](x)f_{\alpha}(x,u),
  \ \ u \in {\dot{H}}^s(\R^{n})
\end{equation}
where $s \in(0,1)$, $0\leq \alpha,\beta<2s<n$, $\mu \in (0,n)$, $\gamma<\gamma_{H}$, $I_{\mu}(x)=|x|^{-\mu}$, $F_{\alpha}(x,u)=\frac{      {|u(x)|}^{ {2^{\#}_{\mu} }(\alpha)}      }{  {|x|}^{  {\delta_{\mu} (\alpha)} }  }$,  $f_{\alpha}(x,u)=\frac{      {|u(x)|}^{{ 2^{\#}_{\mu} }(\alpha)-2}u(x)      }{  {|x|}^{  {\delta_{\mu} (\alpha)}            }  }$, $2^{\#}_{\mu} (\alpha)=(1-\frac{\mu}{2n})\cdot 2^{*}_{s} (\alpha)$,  $\delta_{\mu} (\alpha)=(1-\frac{\mu}{2n})\alpha$,   ${2^{*}_{s}}(\alpha)=\frac{2(n-\alpha)}{n-2s}$ and  $\gamma_{H}=4^s\frac{\Gamma^2(\frac{n+2s}{4})}
{\Gamma^2(\frac{n-2s}{4})}$($\Gamma$ denotes the Gamma function). Intuitively, (\ref{eq1.1}) is
$$
(-\Delta)^{s}u-{\gamma} {\frac{u}{|x|^{2s}}}= {\frac{{|u|}^{ {2^{*}_{s}}(\beta)-2}u}{|x|^{\beta}}}+ \Big( \int_{\R^{n}} \frac{ {|u(y)|}^{ {2^{\#}_{\mu} }(\alpha)} }{ {|x-y|}^{\mu} {|y|}^{  {\delta_{\mu} (\alpha)} } }dy   \Big) \frac{      {|u(x)|}^{{ 2^{\#}_{\mu} }(\alpha)-2}u(x)      }{  {|x|}^{  {\delta_{\mu} (\alpha)}            }  },
  ~~u \in {\dot{H}}^s(\R^{n}).
$$
Noticing that ${2^{*}_{s}}(\alpha)$ is the critical fractional Hardy-Sobolev exponent and  $\gamma_{H}$ is the best fractional Hardy constant on $\R^{n}$ (See Lemmas \ref{lemma2.1}-\ref{lemma2.2}). It is worth while to point out that $\big(2^{\#}_{\mu} (\alpha),\delta_{\mu} (\alpha)\big)$ is a pair of critical exponents in the sense of Fractional Hardy-Sobolev inequality and Hardy-Littlewood-Sobolev inequality, which can be seen later in (\ref{eq2.6}). The fractional Laplacian $(-\Delta)^{s}$ is defined on the Schwartz class (space of rapidly decaying $C^{\infty}$ functions in $\R^{n}$) through Fourier transform,
$$\widehat{(-\Delta)^{s}u}(\xi)= |\xi|^{2s} \hat{u}(\xi), \forall \xi \in \R^{n}      $$
where $\hat{u}(\xi)=\frac{1}{(2\pi)^{n/2}} \int_{\R^{n}} e^{-i\xi x}u(x)dx$ is the Fourier transform of $u$.

Throughout this paper, we denote the norm of $L^{p}(\R^{n},|y|^{-\lambda})$ by $${||u||}_{L^{p}(\R^{n},|y|^{-\lambda})}:=\Big ( \int_{\R^{n}}  \frac{|u(y)|^p}{  |y|^{\lambda} }dy  \Big )^{\frac{1}{p}}$$
for any $0 \leq \lambda<n$ and $p\in[1,+\infty)$. The homogeneous fractional Sobolev space of order $s\in(0,1)$ is defined as
$$\dot{H}^{s}(\R^{n}):=\{u\in  {L}^{2^*_s}(\R^n) : |\xi|^{s}\hat{u}(\xi) \in  {L}^{2}(\R^n)  \},$$
which is in fact the completion of $C_{0}^{\infty}(\R^{n})$ under the norm
$$   ||u||_{\dot{H}^{s}(\R^{n})}^2= \int_{\R^{n}}|\xi|^{2s}|\hat{u}(\xi)|^2d\xi= \int_{\R^{n}} |(-\Delta)^{s/2}u|^2dx.$$
The dual space of ${\dot{H}}^s(\R^{n})$ is denoted by ${{\dot{H}}^s(\R^{n})^{'}}$. See \cite{ENEV} and references therein for the basics on the fractional Laplacian.

We say that $u \in \dot{H}^{s}(\R^{n})$ is a weak solution to $(\ref{eq1.1})$ if
$$ \int_{\R^{n}} \big[(-\Delta)^{\frac{s}{2}}u(-\Delta)^{\frac{s}{2}}{\phi}
-{\gamma} {\frac{u\phi}{|x|^{2s}}}\big]dx=\int_{\R^{n}} {\frac{{|u|}^{ {2^{*}_{s}}(\beta)-2}u\phi}{|x|^{\beta}}}dx+  \int_{\R^{n}}  \big [   I_{\mu}* F_{\alpha}(\cdot,u)  \big](x)f_{\alpha}(x,u)\phi(x)dx  $$
for any $\phi \in \dot{H}^{s}(\R^{n})$. Denote
\begin{equation} \label{eq2.01}
  B_{\alpha}(u,v)=\int_{\R^{n}} \int_{\R^{n}} \frac{{|u(x)|}^{{ 2^{\#}_{\mu} }(\alpha)} {|v(y)|}^{{ 2^{\#}_{\mu} }(\alpha)} }{ {|x|}^{  {\delta_{\mu} (\alpha)}            } {|x-y|}^{\mu} {|y|}^{   {\delta_{\mu} (\alpha)}            } } dx dy,~~~~\forall u,v \in {\dot{H}}^s(\R^{n})
\end{equation}
where $s \in(0,1)$, $0 \leq \alpha<2s<n$, $\mu \in (0,n)$,  $2^{\#}_{\mu} (\alpha)=(1-\frac{\mu}{2n}) 2^{*}_{s} (\alpha)$ and $\delta_{\mu}(\alpha)=(1-\frac{\mu}{2n})\alpha$. In particular, ${2^{*}_{s}}:={2^{*}_{s}}(0)=\frac{2n}{n-2s}$ and $2^{\#}_{\mu}:=2^{\#}_{\mu} (0)=\frac{2n-\mu}{n-2s}$. Set $\tilde{u}_{t}(x)=t^{ \frac{n-2s}{2} }u(tx)$ and $\tilde{v}_{t}(y)=t^{ \frac{n-2s}{2} }v(ty)$, $t>0$, then $B_{\alpha}(\tilde{u}_{t},\tilde{v}_{t})=B_{\alpha}(u,v)$. The energy functional associated to (\ref{eq1.1}) is defined as:
\begin{align*}
  I(u)=\frac{1}{2}\int_{\R^{n}} \big[{|(-\Delta)^{\frac{s}{2}}u|}^2
-{\gamma} {\frac{u^2}{|x|^{2s}}}\big]dx
-\frac{1}{2^{*}_{s}(\beta)}\int_{\R^n}{\frac{{|u|}^{ {2^{*}_{s}}(\beta)}}{|x|^{\beta}}}dx
-\frac{1}{2 \cdot { 2^{\#}_{\mu} }(\alpha) } B_{\alpha}(u,u).
\end{align*}
A nontrivial critical point of $I$ is a nontrivial weak solution to equation (\ref{eq1.1}).

The problem of multiple critical exponents has been extensively studied by scholars, see \cite{RFPP}, \cite{NGSS}, \cite{JYFW}, \cite{NGCY}, \cite{NGFR}, \cite{Ngfr},  \cite{DAEJ}, \cite{DKGL}, \cite{CWJ}, \cite{HYKD} and \cite{JWJP}. Dating back to \cite{RFPP}, R. Filippucci et al. studied the double critical equation of Emden-Fowler type:
\begin{equation} \label{eq1.2}
 -{\Delta}_pu-{\kappa} {\frac{u^{p-1}}{|x|^{p}}}=u^{p^{*}-1}+\frac{u^{ {p^{*}(\alpha)-1}}}{|x|^{\alpha}} ~~~~\mbox{in}~~{\R}^n, u \geq 0, \ \ u \in {D}^{1,p}(\R^{n})
 \end{equation}
where $n\geq 2$, $p \in (1,n)$, $\alpha \in (0,p)$, $p^{*}=\frac{np}{n-p}$, $p^{*}(\alpha)=\frac{p(n-\alpha)}{n-p}$, $0\leq \kappa < \bar{\kappa}=(\frac{n-p}{p})^p$ and ${\Delta}_p u:=div({|\nabla u|}^{p-2}\nabla u)$ is the p-Laplacian of $u$. Their work space ${D}^{1,p}(\R^{n})$ is defined as the completion of $C^{\infty}_{0}({\R}^n)$ under the norm $||u||_{{D}^{1,p}(\R^{n})}=(\int_{\R^{n}} |\nabla u|^p dx)^{\frac{1}{p}}$, i.e.
$${D}^{1,p}(\R^{n}):=\{u\in  {L}^{p^*}(\R^n) : \nabla u \in  {L}^{p}(\R^n)  \}.$$
Through truncation skills, the authors of \cite{RFPP} showed the existence of minimizers for
$$ \bar{\Lambda}(n,\kappa,\alpha)=  \mathop {\inf }\limits_{u \in {D}^{1,p}(\R^{n}) \setminus \{0\}}   \frac{\int_{\R^{n}} |\nabla u|^p dx-{\kappa} \int_{\R^{n}} {\frac{{|u|}^p}{|x|^{p}}}dx}{\Big( \int_{\R^{n}} \frac{{|u|}^{ { p^{*} }(\alpha)}}{|x|^{\alpha}}dx  \Big)^{\frac{p}{  p^{*}(\alpha)  }}} $$
provided $\alpha \in (0,p)$ and $\kappa < \bar{\kappa}$ or $\alpha=0$ and $0\leq \kappa < \bar{\kappa}$. Then they obtain a nontrivial weak solution to problem $(\ref{eq1.2})$ by using mountain pass lemma and a careful analysis of concentration of the corresponding $(PS)$ sequence. 

In \cite{NGFR}, N. Ghoussoub and F. Robert considered the Dirichlet boundary value problem:
\[
\left\{ \begin{gathered}
   -\Delta u-{\gamma} {\frac{u}{|x|^{2}}}=\lambda u+{\frac{{u}^{ {2^{*}}(\alpha)-1}}{|x|^{\alpha}}} {\text{ on }}{\Omega}, \hfill \\
  u > 0{\text{ on }}{\Omega}, \hfill \\
  u = 0{\text{ on }}{\partial\Omega}, \hfill \\
\end{gathered}  \right.
\]
where $\Omega$ is a smooth bounded domain in $\R^n $ such that $0\in\Omega$, $n \geq 3$, $\gamma<\frac{(n-2)^2}{4}$, $0\leq \alpha <2$, $2^{*}(\alpha)=\frac{2(n-\alpha)}{n-2}$, $0\leq \lambda < \lambda_{\gamma}(\Omega)$ and $\lambda_{\gamma}(\Omega)$ is the first
eigenvalue of $-\Delta-{\frac{\gamma}{|x|^{2}}}$ on $H_0^1(\Omega)\setminus\{0\}$. Fruitful achievements have been made in their work. Before long, N. Ghoussoub et al. \cite{Ngfr} extends the results in \cite{NGFR} to nonlocal case.

N. Ghoussoub and  S. Shakerian \cite{NGSS} generalized the results in \cite{RFPP} to $(-\Delta)^{s}$ operater and considered the problem
\begin{equation} \label{eq1.3}
(-\Delta)^{s}u-{\gamma} {\frac{u}{|x|^{2s}}}={|u|}^{ {2^{*}_{s}}-2}u+{\frac{{|u|}^{ {2^{*}_{s}}(\alpha)-2}u}{|x|^{\alpha}}},      \ \ u \in \dot{H}^{s}(\R^{n})
\end{equation}
for $s \in(0,1)$, $0< \alpha<2s<n$ and $0 \leq\gamma<\gamma_{H}$. Through the weighted harmonic extension for the fractional Laplacian obtained by L. Caffarelli and L. Silvestre in \cite{LCLS}, N. Ghoussoub et al. showed the existence of a nontrivial weak solution $w\in X^{s}({\R}^{n+1}_+)$ to
\begin{align*}
\left\{ \begin{gathered}
 -div(y^{1-2s}{\nabla w})=0, \ \ \ \ \ \ \ \ \ \ \ \ \ \ \ \mbox{in}~~{\R}_{+}^{n+1}  \\
\frac{\partial w}{\partial {{\nu}^s}}={\gamma} {\frac{w(x,0)}{|x|^{2s}}}+{|w(x,0)|}^{ {2^{*}_{s}}-2}w(x,0)+{\frac{{|w(x,0)|}^{ {2^{*}_{s}}(\alpha)-2}w(x,0)}{|x|^{\alpha}}}  \  \ \mbox{on}~~{\R}^{n}
\end{gathered} \right.
\end{align*}
where $\frac{\partial w}{\partial {{\nu}}^s}:=-k_{s}\lim_{y \to 0^+ } y^{1-2s}\frac{\partial w(x,y)}{\partial y}$ and the space $X^{s}({\R}^{n+1}_+)$ is defined as the closure of $C_0^{\infty}(\overline{{\R}^{n+1}_+})$ under the norm
$$  {||w||}_{X^{s}({\R}^{n+1}_+)}:=\Big(k_{s}\int_{{\R}_{+}^{n+1}} y^{1-2s}{|\nabla w|}^2dxdy\Big)^{\frac{1}{2}}$$
with $k_{s}=\frac{\Gamma(s)}{2^{1-2s}\Gamma(1-s)}$. Denote the trace of $w(x,y) \in X^{s}({\R}^{n+1}_+)$ on ${\R}^n \times \{y=0\}$ by $w(x,0)$, then $u(x)=w(x,0)$ is in $\dot{H}^{s}(\R^{n})$ and is a weak solution to equation (\ref{eq1.3}).

In \cite{JYFW}, J. Yang and F. Wu studied
\begin{equation} \label{eq1.5}
(-\Delta)^{s}u-{\gamma} {\frac{u}{|x|^{2s}}}={\frac{{|u|}^{ {2^{*}_{s}}(\beta)-2}u}{|x|^{\beta}}}+(I_{\mu}*{|u|}^{ {2^{\#}_{\mu}}}){|u|}^{ {2^{\#}_{\mu}}-2}u, \ \ u \in \dot{H}^{s}(\R^{n})
\end{equation}
where $s\in(0,1)$, $0<\beta<2s<n$, $\mu \in(n-2s,n)$, $\gamma<\gamma_{H}$ , $I_{\mu}(x)={|x|}^{-\mu}$, $2^{\#}_{\mu}=\frac{2n-\mu}{n-2s}$ and  ${2^{*}_{s}}(\beta)=\frac{2(n-\beta)}{n-2s}$. By using the  Nehari manifold method, they proved that equation $(\ref{eq1.5})$ has a nontrivial weak solution if $0<\gamma<\gamma_{H}$. For the cases of the standard Laplacian, biharmonic operator and p-biharmonic operator, the interested reader can refer to \cite{FCZQ}, \cite{JLCS}, \cite{NGAM}, \cite{NGFR}, \cite{NGCY}, \cite{YWYS}, \cite{DAEJ} and \cite{AEKS}.

Motivated by the above papers, we consider the existence of nontrivial weak solutions to problem $(\ref{eq1.1})$.
To the best of our knowledge, $(\ref{eq1.1})$ has not been studied before.

Our main results are as follows:

\begin{theorem} \label{th1.1}
The problem (\ref{eq1.1}) possesses at least a nontrivial weak solution provided either \textbf{(I)} $s \in(0,1)$, $0<\alpha,\beta<2s<n$, $\mu \in (0,n)$ and $\gamma<\gamma_{H}$ \\
or \textbf{(II)} $s \in(0,1)$, $0\leq\alpha,\beta<2s<n$ while $\alpha \cdot \beta=0$, $\mu \in (0,n)$ and $0\leq \gamma<\gamma_{H}$.
\end{theorem}

\begin{remark}
Theorem \ref{th1.1} indicates that we can relax the lower bound of $\gamma$ in equation (\ref{eq1.1}) provided $\alpha,\beta>0$, which is different from equations (\ref{eq1.2})-(\ref{eq1.5}). In the meanwhile, Theorem \ref{th1.1} relaxes the order ${\mu}$ in $I_{\mu}(x)=|x|^{-\mu}$ because equation (\ref{eq1.5}) only allows $\mu \in(n-2s,n)$. Moreover, equation (\ref{eq1.5}) is a special case of equation (\ref{eq1.1}) with $\alpha=0$.
\end{remark}

There are three main difficulties in the proof of Theorem (\ref{th1.1}). Firstly, truncation skills used in \cite{RFPP} and \cite{NGSS} do not work if we choose $\dot{H}^{s}(\R^{n})$ as the work space since $(-\Delta)^{s}$ is a nonlocal operator. Although the weighted harmonic extension can overcome the difficulty of the non-locality of $(-\Delta)^{s}$ if we work in $X^{s}({\R}^{n+1}_+)$, the appearance of the convolution term in (\ref{eq1.1}) still prevents us from using truncation skills. Secondly, the compactness of the corresponding $(PS)$ sequence may not hold since equation (\ref{eq1.1}) has two critical nonlinearities. For the equation with a single critical nonlinearity
\begin{equation} \label{eq1.006}
-\Delta u+\lambda u={|u|}^{2^*-2}u~~~~\mbox{in}~~~~\Omega
\end{equation}
where $\Omega$ is a bounded domain of $\R^{n}$, $n\geq3$, $-{\lambda}_1(\Omega)<\lambda<0$ and $2^*=\frac{2n}{n-2}$, H. Br$\acute{e}$zis and L. Nirenberg in \cite{HBNI} used the Br$\acute{e}$zis-Lieb lemma to prove the compactness of the $(PS)_{\tilde{c}}$ sequence if $\tilde{c}<\tilde{{c}}^*$ where $\tilde{{c}}^*=\frac{1}{n}S^{\frac{n}{2}}$ and $S=\mathop {\inf }\limits_{u \in {D}^{1,2}(\R^{n}) \setminus \{0\}}   \frac{\int_{\R^{n}} |\nabla u|^2 dx}{( \int_{\R^{n}} {|u|}^{ 2^{*} }dx)^{\frac{2}{  2^{*} }}}$.
It seems that the method of \cite{HBNI} does not apply to (\ref{eq1.1}) as the Br$\acute{e}$zis-Lieb type lemma would lead to a system of inequalities which does not have explicit nontrivial solutions. Thirdly, there is an asymptotic competition between the energy carried by the two critical nonlinearities, so we have trouble in ruling out the "vanishing" of the corresponding $(PS)$ sequence.
Naturally, we would hope to overcome this difficulty by using the Nehari manifold method as in \cite{JYFW}, but unfortunately, the corresponding limit equation does not exist since
\begin{equation} \label{eq1.007}
 (-\Delta)^{s}v= \Big( \int_{\R^{n}} \frac{ {|v(y)|}^{ {2^{\#}_{\mu} }(\alpha)} }{ {|x-y|}^{\mu} {|y|}^{  {\delta_{\mu} (\alpha)} } }dy   \Big) \frac{      {|v(x)|}^{{ 2^{\#}_{\mu} }(\alpha)-2}v(x)      }{  {|x|}^{  {\delta_{\mu} (\alpha)}            }  }
\end{equation}
is not translation invariant. To see this, let's go back to equation (\ref{eq1.5}):
$$(-\Delta)^{s}u-{\gamma} {\frac{u}{|x|^{2s}}}={\frac{{|u|}^{ {2^{*}_{s}}(\beta)-2}u}{|x|^{\beta}}}+(I_{\mu}*{|u|}^{ {2^{\#}_{\mu}}}){|u|}^{ {2^{\#}_{\mu}}-2}u$$
which is similar to equation (\ref{eq1.1}), the authors in \cite{JYFW} obtained a nontrivial weak solution to $(\ref{eq1.5})$ by using the Nehari manifold method. The key step was to rule out the "vanishing" of the corresponding $(PS)$ sequence by using the limit equation of $(\ref{eq1.5})$. As can be seen in Section 3 in \cite{JYFW}, there exists a $(PS)$ sequence $\{u_k\}$ for the energy functional corresponding to $(\ref{eq1.5})$ such that $u_k \rightharpoonup u$ in ${\dot{H}}^s(\R^{n})$ with $u$ solving (\ref{eq1.5}). It may occur that $u\equiv0$. Taking ${v}_k(x)=\lambda_k^{ \frac{n-2s}{2} }u_k({\lambda}_k x+x_k)$ where ${\lambda}_k>0$, $x_k \in \R^n$ and $\frac{x_k}{{\lambda}_k} \to \infty $ as $k\to +\infty$, they derived that $v_k \rightharpoonup v$ in ${\dot{H}}^s(\R^{n})$ and
$$ \int_{\R^n}\frac{v_k\phi}{|x+\frac{x_k}{{\lambda}_k}|^{2s}} \to 0,~~~~\int_{\R^n} {\frac{{|v_k|}^{ {2^{*}_{s}}(\beta)-2}v_k\phi}{|x+\frac{x_k}{{\lambda}_k}|^{\beta}}} \to 0~~~~\mbox{as}~~~~ k\to +\infty$$
for any $\phi \in C_0^{\infty}(\R^{n})$. Then $v$ weakly solves
\begin{equation} \label{eq1.04}
(-\Delta)^{s}v=(I_{\mu}*{|v|}^{ {2^{\#}_{\mu}}}){|v|}^{ {2^{\#}_{\mu}}-2}v.
\end{equation}
Using the limit equation (\ref{eq1.04}), they ruled out the "vanishing" of the $(PS)$ sequence for the energy functional corresponding to $(\ref{eq1.5})$. Clearly, this method does not apply to (\ref{eq1.1}) since $(\ref{eq1.007})$ is not translation invariant.

For these reasons, we use a direct way to prove Theorem \ref{eq1.1}. The crucial point is the utilization of the embeddings(See Section 3)
\begin{equation}  \label{eq1.06}
{\dot{H}}^s(\R^{n}) \hookrightarrow  {L}^{2^*_{s}(\alpha)}(\R^{n},|y|^{-\alpha}) \hookrightarrow L^{p,\frac{n-2s}{2}p+pr}(\R^{n},|y|^{-pr})
\end{equation}
for $s \in (0,1)$, $0<\alpha<2s<n$, ${2^{*}_{s}}(\alpha)=\frac{2(n-\alpha)}{n-2s}$, $p\in[1,2^*_{s}(\alpha))$ and $r=\frac{\alpha}{ 2^*_{s}(\alpha) }$, and the following improved Sobolev inequalities:

\begin{proposition}\label{pro1.4}
Let $s \in (0,1)$ and $0<\alpha<2s<n$. Then there exists $C=C(n,s,\alpha)>0$ such that for any $\theta \in (\bar{\theta},1)$ and for any $p\in[1,2^*_{s}(\alpha))$,
there holds
\begin{equation}  \label{eq1.6}
 \Big( \int_{ \R^{n} }  \frac{ |u(y)|^{ 2^*_{s}(\alpha)} }  {  |y|^{\alpha} } dy  \Big)^{ \frac{1}{  2^*_{s} (\alpha)  }}  \leq C ||u||_{{\dot{H}}^s(\R^{n})}^{\theta} ||u||^{1-\theta}_{  L^{p,\frac{n-2s}{2}p+pr}(\R^{n},|y|^{-pr}) },~~~~\forall u \in {\dot{H}}^s(\R^{n})
\end{equation}
where $\bar{\theta}=\max \{ \frac{2}{2^*_{s}(\alpha)}, \frac{2^*_{s}-1}{2^*_{s}(\alpha)}  \} >0$ and $r=\frac{\alpha}{ 2^*_{s}(\alpha) }$.
\end{proposition}

\begin{corollary}\label{coro1.5}
Let $n\geq2$, $1<p<n$ and $0<\alpha<p$. Then there exists $C=C(n,p,\alpha)>0$ such that for any $\theta \in (\bar{\theta},1)$ and for any $m \in[1,p^*(\alpha))$,
there holds
\begin{equation}  \label{eq1.7}
 \Big( \int_{ \R^{n} }  \frac{ |u(y)|^{ p^*(\alpha)} }  {  |y|^{\alpha} } dy  \Big)^{ \frac{1}{  p^*(\alpha)  }}  \leq C ||u||_{{D}^{1,p}(\R^{n})}^{\theta} ||u||^{1-\theta}_{  L^{m,\frac{n-p}{p}m+mr}(\R^{n},|y|^{-mr}) },~~~~\forall u \in {D}^{1,p}(\R^{n})
\end{equation}
where $\bar{\theta}=\max \{ \frac{p}{p^*(\alpha)}, \frac{p^*-1}{p^*(\alpha)}  \} >0$ and $r=\frac{\alpha}{ p^*(\alpha) }$.
\end{corollary}
\begin{remark}
Proposition \ref{pro1.4} and Corollary \ref{coro1.5} are more general than Theorems 1-2 in G. Palatucci, A. Pisante in \cite{GPAP}; The detailed proof will be given in Section 3.
\end{remark}
Now, we give the outline of the proof for Theorem \ref{eq1.1}. We use the Mountain pass lemma to find critical points of $I(u)$ on $\dot{H}^{s}(\R^{n})$, which correspond to weak solutions for equation (\ref{eq1.1}). Since problem (\ref{eq1.1}) includes double critical exponents, we require the Mountain pass level $c<c^*$ for some suitable threshold value $c^*$. This is crucial in  ruling out the "vanishing" of the corresponding (PS) sequence. To this end, we introduce the minimization problems
\begin{equation}  \label{eq1.8}
   S_{\mu}(n,s,\gamma,\alpha)=\mathop {\inf }\limits_{u \in \dot{H}^{s}(\R^{n})) \setminus \{0\}  }   \frac{ \int_{\R^{n}} |(-\Delta)^{s/2}u|^2dx-{\gamma} \int_{\R^{n}} {\frac{u^2}{|x|^{2s}}}dx }
   {  {B_{\alpha}(u,u)}^{\frac{1}{  2^{\#}_{\mu}(\alpha)  }}  }
\end{equation}
and
\begin{equation}  \label{eq1.9}
   \Lambda(n,s,\gamma,\alpha)=  \mathop {\inf }\limits_{u \in \dot{H}^{s}(\R^{n}) \setminus \{0\}}   \frac{\int_{\R^{n}} |(-\Delta)^{s/2}u|^2dx-{\gamma} \int_{\R^{n}} {\frac{u^2}{|x|^{2s}}}dx}{\Big( \int_{\R^{n}} \frac{{|u|}^{ { 2^{*}_{s} }(\alpha)}}{|x|^{\alpha}}dx  \Big)^{\frac{2}{  2^{*}_{s}(\alpha)  }}}
\end{equation}
where $B_{\alpha}(\cdot,\cdot)$ was defined in (\ref{eq2.01}).
Using the minimizers of $S_{\mu}(n,s,\gamma,\alpha)$ and $\Lambda(n,s,\gamma,\alpha)$, we can prove the Mountain pass level $c<c^*$ where
$$
c^*:=\min \Big \{ \frac{2^{\#}_{\mu}(\alpha)-1}{2 \cdot 2^{\#}_{\mu}(\alpha)} {S_{\mu}(n,s,\gamma,\alpha)}^{\frac{{ 2^{\#}_{\mu} }(\alpha)}{ { 2^{\#}_{\mu} }(\alpha)-1 }}, \frac{2s-\beta}{2(n-\beta)} \Lambda(n,s,\gamma,\beta)^{\frac{n-\beta}{2s-\beta}} \Big \}.$$
Then, the Mountain pass lemma gives a $(PS)_c$ sequence $\{u_k\}_{k=1}^{+\infty}$ for $I(\cdot)$ at level $c>0$, i.e.
\begin{equation} \label{eq1.013}
\lim_{k \to +\infty}I(u_k)=c<c^*~~\mbox{and}~~ \lim_{k \to +\infty} I'(u_k)=0~~\mbox{strongly in}~~\dot{H}^{s}(\R^n)'.
\end{equation}
Clearly, $\{u_k\}$ is bounded so we may assume $u_k \rightharpoonup u$ in $\dot{H}^{s}(\R^n)$ for some $u \in \dot{H}^{s}(\R^n)$. However, it may occur that $u\equiv0$. Denote
$$d_1=\lim_{k \to +\infty}\int_{\R^n}{\frac{{|u_k|}^{ {2^{*}_{s}}(\beta)}}{|x|^{\beta}}}dx,~~~~d_2=\lim_{k \to +\infty}B_{\alpha}(u_k,u_k).$$
From $(\ref{eq1.8})$, $(\ref{eq1.9})$ and $(\ref{eq1.013})$, we have
\begin{equation} \label{eq1.014}
  d_1^{\frac{2}{  2^{*}_{s}(\beta) }}A_1 \leq d_2,~~~~d_2^{\frac{1}{  2^{\#}_{\mu}(\alpha) }}A_2\leq d_1
\end{equation}
where $A_1=\Lambda(n,s,\gamma,\beta)-  [\frac{2(n-\beta)}{2s-\beta} c]^{\frac{2^{*}_{s}(\beta)-2}{  2^{*}_{s}(\beta) }}$ and $A_2=S_{\mu}(n,s,\gamma,\alpha)-  [\frac{2 \cdot 2^{\#}_{\mu}(\alpha)}{2^{\#}_{\mu}(\alpha)-1} c]^{\frac{2^{\#}_{\mu}(\alpha)-1}{  2^{\#}_{\mu}(\alpha) }}$. Since $c<c^*$, we derive that $A_1>0$ and $A_2>0$. Thus $(\ref{eq1.014})$ implies that $d_1\geq{\varepsilon}_0>0$ and $d_2\geq{\varepsilon}_0>0$(If $d_1=0$ and $d_2=0$, then $c=0$, a contradiction), i.e.
$$\lim_{k \to +\infty}\int_{\R^n}{\frac{{|u_k|}^{ {2^{*}_{s}}(\beta)}}{|x|^{\beta}}}dx\geq{\varepsilon}_0>0,~~~~~~~~\lim_{k \to +\infty}B_{\alpha}(u_k,u_k)\geq{\varepsilon}_0>0. $$
So the embeddings $(\ref{eq1.06})$ and the improved Sobolev inequality (\ref{eq1.6}) imply that
$$   0<C \leq ||u_k||_{  L^{2,{n-2s}+2r}(\R^{n},|y|^{-2r}) } \leq C^{-1}~~~~~~ \mbox{for any}~~~~ k\geq K~~~~ \mbox{large} $$
where $r=\frac{\alpha}{ 2^*_{s}(\alpha) }$ and $C>0$ is a constant.
For any $k\geq K$ large, we may find ${\lambda}_k>0$ and $x_k \in \R^{n}$ such that
$$  {\lambda}_k^{-2s+2r} \int_{B_{{\lambda}_k}(x_k)} \frac{|u_k(y)|^2}{  |y|^{2r} }dy > ||u_k||^2_{  L^{2,{n-2s}+2r}(\R^{n},|y|^{-2r}) } -\frac{C}{2k} \geq C_1>0.$$
Let $v_k(x)={\lambda}_k^{ \frac{n-2s}{2} }u_k({\lambda}_kx)$, then we have ${v}_k \rightharpoonup v\not \equiv 0$ in ${\dot{H}}^s(\R^{n})$. In fact, we can prove that $\{{\tilde{x}}_k=\frac{x_k}{{\lambda}_k}\}$ is bounded and
\begin{equation}\label{eq1.015}
\int_{B_{1}({\tilde{x}}_k)} \frac{|v_k(x)|^2}{  |x|^{2r} }dx  \geq C_1>0.
\end{equation}
From (\ref{eq1.015}), we have $\int_{B_{1}({\tilde{x}}_k)} \frac{|v(x)|^2}{  |x|^{2r} }dx  \geq C_1>0$ since $r<s$.
Moreover, we can check that $\{{v}_k\}$ is a new $(PS)$ sequence for $I(\cdot)$ at the same energy level $c$, then $v\not \equiv 0$ solves (\ref{eq1.1}).\\

It remains to deal with the minimization problems (\ref{eq1.8})-(\ref{eq1.9}). To this end, we need some kind of compactness. When $\alpha=0$, we use the method introduced by R. Filippucci et al. in \cite{RFPP} or S. Dipierro et al. in \cite{SDLM} to prove the existence of minimizers for $S_{\mu}(n,s,\gamma,0)$. Next, we focus on the case of $\alpha>0$. Both N. Ghoussoub et al. in \cite{NGSS} and R. Filippucci et al. in \cite{RFPP} use truncation skills and a careful analysis of concentration to eliminate the "vanishing" of the corresponding minimizing sequence. Therefore, it would inevitably lead to tedious and complex calculations. In addition, the authors in \cite{NGSS} and \cite{CWJ} had to work in the extension space $X^{s}({\R}^{n+1}_+)$ to deal with the non-local operator $(-\Delta)^{s}$. \textbf{If $\alpha>0$, the embeddings (\ref{eq1.06}) and the inequality (\ref{eq1.6}) allow us to adopt a direct but easier way to prove the existence of minimizers for $S_{\mu}(n,s,\gamma,\alpha)$ and ${\Lambda}(n,s,\gamma,\alpha)$ in $\dot{H}^{s}(\R^{n})$. Moreover, (\ref{eq1.06}) and (\ref{eq1.6}) are very useful to rule out the "vanishing" of the corresponding $(PS)$ sequence}. As far as we know, the strategy we adopt is new.  Neither do we use truncation skills nor do we work in $X^s({\R}^{n+1}_+)$, consequently our strategy avoids tedious and complex calculations, and does enormously simplify the proof of the main results in \cite{NGSS} and \cite{CWJ}. To go further, Corollary \ref{coro1.5} and the corresponding embeddings can be applied to solve equation (\ref{eq1.2})
$$
 -{\Delta}_pu-{\kappa} {\frac{u^{p-1}}{|x|^{p}}}=u^{p^{*}-1}+\frac{u^{ {p^{*}(\alpha)-1}}}{|x|^{\alpha}} ~~~~\mbox{in}~~{\R}^n, u \geq 0, \ \ u \in {D}^{1,p}(\R^{n})
$$
where $n\geq 2$, $p \in (1,n)$, $\alpha \in (0,p)$, $p^{*}=\frac{np}{n-p}$, $p^{*}(\alpha)=\frac{p(n-\alpha)}{n-p}$, $0\leq \kappa < \bar{\kappa}=(\frac{n-p}{p})^p$. We also notice that (\ref{eq1.06}) and (\ref{eq1.6}) play the same role as concentration compactness principle does in \cite{NGSS} and \cite{CWJ}. For more information about the concentration compactness principle, please refer to \cite{PLLI} and \cite{PLLO}.

The rest of the paper is organized as follows: in Section 2, we give some preliminaries. In Section 3, we introduce the weighted Morrey space and establish improved Sobolev inequalities, i.e., we prove Proposition  \ref{pro1.4} and Corollary \ref{coro1.5}. In Section 4, we solve the minimization problems (\ref{eq1.8})-(\ref{eq1.9}). In Section 5, we prove Theorem \ref{th1.1}. \\

\textbf{Notation:} We use $\rightarrow$ and $\rightharpoonup$ to denote the strong and weak convergence in the corresponding spaces respectively. Write "Palais-Smale" as $(PS)$ in short. $\mathbb{N}=\{1,2,\cdots\}$ is the set of natural numbers. $\mathbb{R}$ and $\mathbb{C}$ denote the sets of real and complex numbers respectively. By saying a function is "measurable", we always mean that the function is "Lebesgue" measurable. "$\wedge$" denotes the Fourier transform and "$\vee$" denotes the inverse Fourier transform. Generic fixed and numerical constants will be denoted by $C$(with subscript in some case) and they will be allowed to vary within a single line or formula.

\section{Preliminaries}
\setcounter{equation}{0}
In this section, we give some preliminary results.

\begin{lemma}\label{lemma2.1} (Fractional Hardy inequality: Formula (2.1) in \cite{EHRS}) \\
Let $s\in (0,1)$ and $n>2s$. Then we have
\begin{equation}\label{eq2.1}
  {\gamma}_{H} \int_{\R^{n}} {\frac{u^2}{|x|^{2s}}}dx \leq \int_{\R^{n}} |(-\Delta)^{s/2}u|^2dx, ~~~~ \forall u \in {\dot{H}}^s(\R^{n})
\end{equation}
where $\gamma_{H}:=4^s\frac{\Gamma^2(\frac{n+2s}{4})}{\Gamma^2(\frac{n-2s}{4})}$ is the best constant in the above inequality on $\R^{n}$.  \\
\end{lemma}

\begin{lemma}\label{lemma2.2} (Fractional Hardy-Sobolev inequalities: Lemma 2.1 of \cite{NGSS}) \\
Let $s\in (0,1)$ and $0\leq \alpha \leq 2s<n.$ Then there exist positive constants $c$ and $C$ such that
\begin{equation}\label{eq2.2}
   \Big( \int_{\R^{n}} \frac{{|u|}^{ { 2^{*}_{s} }(\alpha)}}{|x|^{\alpha}}dx  \Big)^{\frac{2}{  2^{*}_{s}(\alpha)  }} \leq c \int_{\R^{n}} |(-\Delta)^{s/2}u|^2dx,~~~~\forall u \in {\dot{H}}^s(\R^{n}).
\end{equation}
Moreover, if $\gamma<\gamma_{H}=4^s\frac{\Gamma^2(\frac{n+2s}{4})}{\Gamma^2(\frac{n-2s}{4})}$, then
\begin{equation}\label{eq2.3}
  C \Big( \int_{\R^{n}} \frac{{|u|}^{ { 2^{*}_{s} }(\alpha)}}{|x|^{\alpha}}dx  \Big)^{\frac{2}{  2^{*}_{s}(\alpha)  }} \leq \int_{\R^{n}} |(-\Delta)^{s/2}u|^2dx-{\gamma} \int_{\R^{n}} {\frac{u^2}{|x|^{2s}}}dx,~~~~\forall u \in {\dot{H}}^s(\R^{n}).
\end{equation}  \\
\end{lemma}

From Lemma \ref{lemma2.1}, the following inequality holds for all $\gamma<\gamma_{H}$ and any $u \in {\dot{H}}^s(\R^{n})$,
\begin{equation}\label{eq2.4}
  (1-\frac{\gamma_{+}}{\gamma_{H}})\int_{\R^{n}} |(-\Delta)^{s/2}u|^2dx \leq ||u||^2 \leq (1+\frac{\gamma_{-}}{\gamma_{H}}) \int_{\R^{n}} |(-\Delta)^{s/2}u|^2dx
\end{equation}
where $||u||={\Big(\int_{\R^{n}} |(-\Delta)^{s/2}u|^2dx-{\gamma} \int_{\R^{n}} {\frac{u^2}{|x|^{2s}}}dx \Big)}^{\frac{1}{2}}$ and $\gamma_{\pm}=\max\{\pm\gamma,0\}$. We define an \textbf{equivalent norm} on ${\dot{H}}^s(\R^{n})$ by $||\cdot||$ and denote the inner product of $u,v \in {\dot{H}}^s(\R^{n})$ by
$$ {\langle u, v \rangle}=\int_{\R^{n}} (-\Delta)^{\frac{s}{2}}u(-\Delta)^{\frac{s}{2}}vdx-{\gamma} \int_{\R^{n}} {\frac{uv}{|x|^{2s}}}dx.$$ \\

\begin{lemma}  \label{lemma2.3}
 Let $s \in (0,1)$ and $0<r<s<\frac{n}{2}$. If $\{u_k\}$ is a bounded sequence in ${\dot{H}}^s(\R^{n})$ and $u_k \rightharpoonup   u ~~\mbox{in}~~{\dot{H}}^s(\R^{n})$, then
$$\frac{u_k}{ {|x|}^{r}  }    \rightarrow   \frac{u}{ {|x|}^{r}  } ~~\mbox{in}~~L^2_{loc}(\R^{n}).$$
\end{lemma}

\begin{proof}
Since $u_k \rightharpoonup   u ~~\mbox{in}~~{\dot{H}}^s(\R^{n})$, by Corollary 7.2 in \cite{ENEV}, we have
\begin{align*}
   u_k \rightarrow   u ~~\mbox{in}~~L^q_{loc}(\R^{n})(1 \leq q< 2^*_s)~~~~\mbox{and}~~~~
  u_k \rightarrow   u ~~\mbox{a.e.~~~~on}~~\R^{n}.
\end{align*}
From Lemma \ref{lemma2.1}, we have
$\int_{\R^{n}}\frac{|u_k|^2}{|x|^{2s}}dx \leq C_{s,n} \int_{\R^{n}} |(-\Delta)^{\frac{s}{2} }u_k|^2 dx \leq \tilde{C}.$
For any compact set $\Omega \subset \R^{n}$, using H\"older's inequality, we have
\begin{align*}
\int_{\Omega}\frac{|u_k-u|^2}{|x|^{2r}}dx &
 \leq \Big(\int_{\Omega}\frac{|u_k-u|^2}{|x|^{2s}}dx \Big)^{\frac{r}{s}}  \Big(\int_{\Omega}|u_k-u|^2dx \Big)^{(1-\frac{r}{s})}  \\
& \leq C \Big(\int_{\Omega}|u_k-u|^2dx \Big )^{(1-\frac{r}{s})} \rightarrow 0 .
\end{align*}
\end{proof}

\begin{proposition} (Hardy-Littlewood-Sobolev inequality, Theorem 4.3 in \cite{ELMA})  \label{prop2.4}
 Let $t, r >1$ and $\mu \in (0,n)$ with $\frac{1}{t}+\frac{\mu}{n}+\frac{1}{r}=2$, $f\in L^t(\R^{n})$ and $h\in L^r(\R^{n})$. There exists a sharp constant $C(t,n,\mu,r)$, independent of $f$, $h$ such that
 \begin{equation}\label{eq2.5}
 \Big | {\int_{\R^{n}} \int_{\R^{n}} \frac{f(x)h(y) }{  {|x-y|}^{\mu}  } dx dy} \Big | \leq C(t,n,\mu,r) {||f||}_{L^t(\R^{n})}{||h||}_{L^r(\R^{n})}.
 \end{equation}
If $t=r=\frac{2n}{2n-\mu}$, then $ C(t,n,\mu,r)=C(n,\mu)={\pi}^{\frac{\mu}{2}} \frac{ \Gamma(\frac{n}{2}-\frac{\mu}{2})
 }{ \Gamma(n-\frac{\mu}{2})   } \Big \{   {\frac{ \Gamma(\frac{n}{2})
 }{ \Gamma(n)}  \Big \}^{-1+\frac{\mu}{n}}
 }.$
 In this case there is equality in (\ref{eq2.5}) if and only if $f\equiv(\mbox{constant})h$ and $ h(x)=A({\varepsilon}^2+{|x-a|}^2)^{ \frac{-(2n-\mu)}{2} }$ for some $A\in\mathbb{C}$, $0\not =\varepsilon \in\mathbb{R}$ and $a\in\mathbb{R}^n$. \\
\end{proposition}

Let $s \in(0,1)$, $0 \leq \alpha<2s<n$, $\mu \in (0,n)$. $\forall u \in \dot{H}^{s}(\R^{n})$, take $t=r=\frac{2n}{2n-\mu}>1$ and
$f(\cdot)=h(\cdot)=\frac{      {|u(\cdot)|}^{{ 2^{\#}_{\mu} }(\alpha)}      }{  {|\cdot|}^{  {\delta_{\mu} (\alpha)}            }  }$
in (\ref{eq2.5}). Then Lemma \ref{lemma2.2} implies that $f, h\in L^{\frac{2n}{2n-\mu}}(\R^{n})$ and for the $B_{\alpha}(\cdot,\cdot)$ introduced in (\ref{eq2.01}), we have
\begin{equation} \label{eq2.6}
  B_{\alpha}(u,u) \leq C(n,\mu)  \Big( \int_{\R^{n}} \frac{{|u|}^{ { 2^{*}_{s} }(\alpha)}}{|x|^{\alpha}}dx  \Big)^{ \frac{2n-\mu}{n} }
 \leq C  ||u||_{{\dot{H}}^s(\R^{n})} ^{  2 \cdot 2^{\#}_{\mu} (\alpha) },~~~~~~~~\forall u \in \dot{H}^{s}(\R^{n}).
\end{equation}  \\

\begin{lemma} (A variant of Brezis-Lieb lemma)  \label{lemma2.5}\\
Let $r>1$, $q\in[1,r]$ and $\delta \in[0,nq/r)$. Assume $\{ w_k\}$ is a bounded sequence in $L^{r}(\mathbb{R}^n,|x|^{-{\delta r}/q})$ and $ w_k\rightarrow w$ a.e. on $\R^{n}$. Then,
\begin{align*}
& \mathop {\lim }\limits_{k  \to \infty}  \int_{\R^{n}}
{\Big| \frac{{|w_k|}^q}{|x|^{\delta}}-\frac{{|w_k-w|}^q}{|x|^{\delta}}
-\frac{{|w|}^q}{|x|^{\delta}}\Big|}^{\frac{r}{q}}=0,\\
& \mathop {\lim }\limits_{k  \to \infty}  \int_{\R^{n}}
{\Big|\frac{{|w_k|}^{q-1}w_k}{|x|^{\delta}}
-\frac{{|w_k-w|}^{q-1}(w_k-w)}{|x|^{\delta}}
-\frac{{|w|}^{q-1}w}{|x|^{\delta}}\Big|}^{\frac{r}{q}}=0.
\end{align*}
\end{lemma}
\begin{proof}
For the case of $\delta=0$, one can refer to Lemma 2.3 in \cite{GSIN}; We focus on the case of $\delta>0$. Fix $\varepsilon>0$ small, there exists $C(\varepsilon)>0$ such that for all $a,b\in\mathbb{R}$ we have
\begin{align*}
{\Big|{{|a+b|}^q}-{{|a|}^q}\Big|} \leq \varepsilon {|a|}^q+C(\varepsilon){|b|}^q,~~~~
{\Big|{{|a+b|}^{q-1}(a+b)}-{{|a|}^{q-1}a}\Big|} \leq \varepsilon {|a|}^q+C(\varepsilon){|b|}^q.
\end{align*}
Using the inequality $(a+b)^p\leq 2^{p-1}(a^p+b^p)$ for $a,b\geq0$ and $p\geq 1$, we obtain
\begin{align} \label{eq2.8}
{\Big|{{|a+b|}^q}-{{|a|}^q}\Big|}^{\frac{r}{q}} \leq  \Big(\varepsilon {|a|}^q+C(\varepsilon){|b|}^q \Big)^{\frac{r}{q}}
\leq  \tilde{\varepsilon} {|a|}^r+\tilde{C}(\varepsilon){|b|}^r
\end{align}
and
\begin{align} \label{eq2.9}
{\Big|{{|a+b|}^{q-1}(a+b)}-{{|a|}^{q-1}a}\Big|}^{\frac{r}{q}} \leq  \Big(\varepsilon {|a|}^q+C(\varepsilon){|b|}^q \Big)^{\frac{r}{q}}
\leq  \tilde{\varepsilon} {|a|}^r+\tilde{C}(\varepsilon){|b|}^r
\end{align}
where $\tilde{\varepsilon}=2^{\frac{r}{q}-1}\varepsilon^{\frac{r}{q}}$ and $\tilde{C}(\varepsilon)= 2^{\frac{r}{q}-1} C(\varepsilon)^{\frac{r}{q}}$. Taking $a=\frac{w_k-w}{|x|^{{\delta}/{q}}}$, $b=\frac{w}{|x|^{{\delta}/{q}}}$ in (\ref{eq2.8}) and (\ref{eq2.9}) respectively. The rest is similar to the proof of Lemma 2.3 in \cite{GSIN}, we omit the details.

\end{proof}

\begin{lemma} (Weak Young inequality, Section 4.3 in \cite{ELMA}) \label{lemma2.6}
Let $n \in \mathbb{N}$, $\mu \in (0,n)$, $\hat{p},\hat{r}>1 $ and $\frac{1}{\hat{p}}+\frac{\mu}{n}=1+\frac{1}{\hat{r}}$. If $ v\in L^{\hat{p}}(\R^{n})$, then $ I_{\mu}*v \in L^{\hat{r}}(\R^{n})$ and
\begin{align} \label{eq2.15}
 \Big(\int_{\R^{n}}{ |{I_{\mu}*v}|}^{\hat{r}} \Big)^{\frac{1}{\hat{r}}} \leq C(n,\mu, \hat{p}) \Big(\int_{\R^{n}}{ |v|}^{\hat{p}} \Big)^{\frac{1}{\hat{p}}}
\end{align}
where $I_{\mu}(x)=|x|^{-\mu}$. In particular, we can set $\hat{r}={\frac{n{\hat{p}}}{n-(n-\mu){\hat{p}}}}$ for $\hat{p}\in(1, \frac{n}{n-\mu})$.   \\
\end{lemma}

\begin{lemma} (Brezis-Lieb type lemma, Lemma 2.4 in \cite{VMJV}) \label{lemma2.7}
Let $n \in \mathbb{N}$, $\mu \in (0,n)$, $\frac{2n-\mu}{2n}\leq p <\infty$ and $\{u_k\}_{k \in \mathbb{N}}$ be a bounded sequence in $L^{\frac{2np}{2n-\mu}}(\R^{n})$. If $u_k \rightarrow u$ a.e. on $\mathbb{R}^n$ as $k\to \infty$, then
\begin{align} \label{eq2.17}
& \mathop {\lim }\limits_{k  \to \infty}  \int_{\R^{n}} \big [\big(I_{\mu}*{|u_k|}^{p}\big){|u_k|}^{p} -\big(I_{\mu}*{|u_k-u|}^{p}\big){|u_k-u|}^{p} \big]= \int_{\R^{n}} \big(I_{\mu}*{|u|}^{p}\big){|u|}^{p}.
\end{align}  \\
\end{lemma}

\begin{lemma}\label{lemma2.8}
Let $s \in(0,1)$, $0\leq\alpha<2s<n$ and $\mu \in (0,n)$. If $\{u_k\}_{k \in \mathbb{N}}$ is a bounded sequence in $\dot{H}^{s}(\R^{n})$ and $u_k \rightharpoonup u$ in $\dot{H}^{s}(\R^{n})$, then we have
\begin{align*}
\mathop {\lim }\limits_{k  \to \infty} B_{\alpha}(u_k,u_k)=\mathop {\lim }\limits_{k  \to \infty} B_{\alpha}(u_k-u,u_k-u)+B_{\alpha}(u,u)
\end{align*}
where $B_{\alpha}(\cdot,\cdot)$ was defined in (\ref{eq2.01}).
\end{lemma}
\begin{proof}
For $s \in(0,1)$, $0\leq\alpha<2s<n$ and $\mu \in (0,n)$, we can check that $\frac{2n-\mu}{2n}<1<2^{\#}_{\mu} (\alpha)$. Therefore, taking $p=2^{\#}_{\mu} (\alpha)$ in Lemma \ref{lemma2.7}, we have $\frac{2np}{2n-\mu}=2^*_{s}(\alpha)$. Since $ u_k \in \dot{H}^{s}(\R^{n})$ and $u_k \rightharpoonup u$ in $\dot{H}^{s}(\R^{n})$, the embedding ${\dot{H}}^s(\R^{n}) \hookrightarrow  {L}^{2^*_{s}(\alpha)}(\R^{n},|x|^{-\alpha})$ in Lemma \ref{lemma2.2} implies that
$$ \frac{u_k}{ |x|^{\frac{\alpha}{2^*_{s}(\alpha)}} },~~~~~~\frac{u}{ |x|^{\frac{\alpha}{2^*_{s}(\alpha)}} } \in {L}^{2^*_{s}(\alpha)}(\R^{n}),$$
$$\frac{u_k}{ |x|^{\frac{\alpha}{2^*_{s}(\alpha)}} } \rightarrow \frac{u}{ |x|^{\frac{\alpha}{2^*_{s}(\alpha)}} } ~~~~\mbox{a.e.~~on}~~~~\mathbb{R}^n.$$
Consequently, Lemma \ref{lemma2.7} gives the desired equality.
\end{proof}

\begin{lemma}   \label{lemma2.9}
Let $s \in(0,1)$, $0 \leq \alpha<2s<n$, $\mu \in (0,n)$ and $\{u_k\}_{k \in \mathbb{N}}$ be a bounded sequence in ${L}^{2^*_{s}(\alpha)}(\R^{n},|x|^{-\alpha})$. If $u_k \rightarrow u$ a.e. on $\mathbb{R}^n$ as $k\to \infty$, then for any $\phi \in {L}^{2^*_{s}(\alpha)}(\R^{n},|x|^{-\alpha})$ we have
\begin{align} \label{eq2.22}
\mathop {\lim }\limits_{k  \to \infty}  \int_{\R^{n}}
\big [   I_{\mu}* F_{\alpha}(\cdot,u_k)  \big](x)f_{\alpha}(x,u_k)\phi(x)dx
= \int_{\R^{n}}
\big [   I_{\mu}* F_{\alpha}(\cdot,u)  \big](x)f_{\alpha}(x,u)\phi(x)dx
\end{align}
where $F_{\alpha}$ and $f_{\alpha}$ were introduced in (\ref{eq1.1}).
\end{lemma}

\begin{proof}
Since $\phi={\phi}^{+}-{\phi}^{-}$, we just consider $\phi\geq0$. For $n \in \mathbb{N}$, denote ${\tilde{u}_k}=u_k-u$, we rewrite the left hand side of (\ref{eq2.22}) as
\begin{align*}
& \int_{\R^{n}}
\big [   I_{\mu}* F_{\alpha}(\cdot,u_k)  \big](x)f_{\alpha}(x,u_k)\phi(x)dx  \\
=&\int_{\R^{n}} \big [   I_{\mu}*\big( F_{\alpha}(\cdot,u_k)-F_{\alpha}(\cdot,\tilde{u}_k) \big) \big] (x)f_{\alpha}(x,u_k)\phi(x)dx   \\
&+\int_{\R^{n}} \big [   I_{\mu}*\big( f_{\alpha}(\cdot,u_k)\phi-f_{\alpha}(\cdot,\tilde{u}_k)\phi \big) \big] (x)F_{\alpha}(x,\tilde{u}_k)dx  \\
&+\int_{\R^{n}}
\big [   I_{\mu}* F_{\alpha}(\cdot,\tilde{u}_k)  \big](x)f_{\alpha}(x,\tilde{u}_k)\phi(x)dx:=\tilde{B}_1+\tilde{B}_2+\tilde{B}_3.
\end{align*}
Denote $p=2^{\#}_{\mu} (\alpha)$ in this Lemma. Apply Lemma \ref{lemma2.5} with $(r,q,\delta)=(\frac{2np}{2n-\mu},p,\delta_{\mu}(\alpha))$ by taking respectively $(w_n,w)=(u_n,u)$ and then $(w_n,w)=(u_n{\phi}^{\frac{1}{p}},u{\phi}^{\frac{1}{p}})$ , and Lemma \ref{lemma2.6} with $\hat{p}=\frac{2n}{2n-\mu}$, we can complete the proof by imitating the argument of Lemma 2.4 in \cite{GSIN}.

\end{proof}

\section{proof of Proposition  \ref{pro1.4} and Corollary \ref{coro1.5}}
In this section, we give some basic properties of a weighted Morrey space and then prove Proposition  \ref{pro1.4} and Corollary \ref{coro1.5}.

The Morrey spaces were introduced by C. Morrey in 1938 \cite{CBMO} to investigate the local behavior of solutions
to some partial differential equations. Nowadays the Morrey spaces were extended to more general cases(see \cite{GPAP}, \cite{YKSS} and \cite{YSAW}). Let $p\in[1,+\infty)$ and $\gamma\in (0,n)$, the usual homogeneous Morrey space
$$L^{p,\gamma}(\R^{n})=\Big \{ u:  ||u||_{  L^{p,\gamma}(\R^{n}) } < +\infty \Big \}$$
was introduced in \cite{GPAP} with the norm
$$ ||u||_{  L^{p,\gamma}(\R^{n}) }=\mathop {\sup }\limits_{R>0,x \in \R^{n}} \Big \{R^{\gamma-n} \int_{B_R(x)} |u(y)|^pdy  \Big \}^{\frac{1}{p}} .$$
One can see that if $\gamma=n$ then $L^{p,\gamma}(\R^{n})$ coincide with $L^p(\R^{n})$ for any $p\geq1$; Similarly $L^{p,0}(\R^{n})$ coincide with $L^{\infty}(\R^{n})$.

Here we mainly state a special weighted Morrey space $L^{p,\gamma +\lambda}(\R^{n},|y|^{-\lambda})$, which was used in \cite{YKSS} and \cite{YSAW}. For $p\in[1,+\infty)$, $\gamma, \lambda>0$ and $\gamma+\lambda \in (0,n)$, we say a Lebesgue measurable function $u: \R^{n}\rightarrow \R$ belongs to $L^{p,\gamma +\lambda}(\R^{n},|y|^{-\lambda})$ if
$$ ||u||_{  L^{p,\gamma +\lambda}(\R^{n},|y|^{-\lambda}) }=\mathop {\sup }\limits_{R>0,x \in \R^{n}} \Big \{R^{\gamma+\lambda-n} \int_{B_R(x)} \frac{|u(y)|^p}{  |y|^{\lambda} }dy  \Big \}^{\frac{1}{p}} < +\infty.$$
Then the following fundamental properties \textbf{(1)}-\textbf{(5)} hold via H\"older's inequality:

\textbf{(1)}~$L^{p\rho}(\R^{n},|y|^{-\rho\lambda}) \hookrightarrow L^{p,\gamma +\lambda}(\R^{n},|y|^{-\lambda})$ for $\rho=\frac{n}{\gamma +\lambda}>1$.

\textbf{(2)}~For any $p\in(1,+\infty)$, we have $L^{p,\gamma +\lambda}(\R^{n},|y|^{-\lambda}) \hookrightarrow L^{1,\frac{\gamma}{p}+ \frac{\lambda}{p} }(\R^{n},|y|^{-\frac{\lambda}{p}}) .$

\textbf{(3)}~Take $\gamma+\lambda=n$, we get $L^{p}(\R^{n},|y|^{-\lambda})$.  \\
Moreover, if we assume $s\in (0,1)$ and $0<\alpha<2s<n$, then we have

\textbf{(4)}~For any $p\in[1,2^*_{s}(\alpha))$, ${\dot{H}}^s(\R^{n}) \hookrightarrow L^{2^*_{s}(\alpha)}(\R^{n},|y|^{-\alpha}) \hookrightarrow L^{p,\frac{n-2s}{2}p+pr}(\R^{n},|y|^{-pr})$ with $r=\frac{\alpha}{2^*_{s}(\alpha)}$ and the three norms in these spaces share the same dilation invariance.

\textbf{(5)}~For any $p\in[1,2^*_{s})$, ${\dot{H}}^s(\R^{n}) \hookrightarrow L^{2^*_{s}}(\R^{n}) \hookrightarrow L^{p,\frac{n-2s}{2}p}(\R^{n})$, refer to page 815 in \cite{GPAP}. \\

\begin{lemma} (Theorem 1 in \cite{ESRW}, or Theorem D in \cite{BMRW}) \label{lemma3.4}
Suppose that $0<\tilde{s}<n$, $1 < \tilde{p} \leq \tilde{q} <+\infty$, $\tilde{p}'=\frac{\tilde{p}}{\tilde{p}-1}$ and that $V$ and $W$ are nonnegative measurable functions on $\R^{n}$, $n \geq 1$. If for some $\sigma>1$
\begin{equation}\label{eq3.3}
|Q|^{ \frac{\tilde{s}}{n}+\frac{1}{\tilde{q}}-\frac{1}{\tilde{p}}} { \Big(  \frac{1}{|Q|} \int_{Q}V^{\sigma}dy    \Big) }^{ \frac{1}{{\tilde{q}} \sigma} } { \Big(  \frac{1}{|Q|} \int_{Q}W^{(1-{\tilde{p}}')\sigma}dy    \Big) }^{ \frac{1}{ {\tilde{p}}' \sigma} } \leq C_{\sigma}
\end{equation}
for all cubes $Q \subset {\R}^n$, then for any function $f \in L^{\tilde{p}}({\R}^n,W(y))$ we have
\begin{equation}\label{eq3.4}
 { \Big(   \int_{{\R}^n}|{\ell}_{\tilde{s}}f(y)|^{\tilde{q}}V(y)dy    \Big) }^{ \frac{1}{\tilde{q}} }  \leq C C_{\sigma} { \Big(   \int_{{\R}^n}|f(y)|^{\tilde{p}}W(y)dy    \Big) }^{ \frac{1}{\tilde{p}} }
\end{equation}
where $C=C(\tilde{p},\tilde{q},n)$ and  ${\ell}_{\tilde{s}}f$ denotes the Riesz potential of order $\tilde{s}$, namely
\begin{equation}\label{eq3.2}
 {\ell}_{\tilde{s}}f(y)=\int_{{\R}^n} \frac{f(z)}{|y-z|^{n-{\tilde{s}}}}dz.
\end{equation}
\end{lemma}
\begin{remark} One can refer to \cite{GPAP} for more information about the Riesz potential. \\
\end{remark}

\noindent\textbf{Proof of  Proposition \ref{pro1.4}}

For $u \in {\dot{H}}^s(\R^{n})$, we have $\hat{g}(\xi):=|\xi|^s\hat{u}(\xi) \in L^2(\R^{n})$ and
$||u||_{{\dot{H}}^s(\R^{n})}=||g||_{L^2(\R^{n})}$  by Plancherel's theorem. Thus, $u(x)=(\frac{1}{|\xi|^s})^{\vee}*g(x)={\ell}_sg(x)$, where ${\ell}_sg(x)=\int_{{\R}^n} \frac{g(z)}{|x-z|^{n-{s}}}dz$.

Firstly, take $\tilde{s}=s$, $\tilde{p}=2$, $\max \{ {2, 2^*_{s}-1} \} < {\tilde{q}}< 2^*_{s}(\alpha)$, $W(y)\equiv 1$, $V(y)=\frac{  {|u(y)|}^  { 2^{*}_{s}(\alpha)-{\tilde{q}}}    }   {  |y|^{\alpha}  }$ and $\sigma=\frac{1}{2^*_{s}-{\tilde{q}}}>1$ in Lemma \ref{lemma3.4}, then (\ref{eq3.3}) becomes
\begin{equation}\label{eq3.5}
|Q|^{ \frac{s}{n}+\frac{1}{\tilde{q}}-\frac{1}{2}} { \Big(  \frac{1}{|Q|} \int_{Q}V^{\sigma}dy    \Big) }^{ \frac{1}{{\tilde{q}} \sigma} }  \leq C_{\sigma}.
\end{equation}

Secondly, we verify condition (\ref{eq3.3}). For any fixed $x\in {\R}^n$, replacing $Q$ by ball $B_{R}(x)$, since $0<[2^{*}_{s}(\alpha)-{\tilde{q}}]\sigma
<1$ and $\frac{ t \sigma {\alpha}}{ 1- [{2^{*}_{s}(\alpha)}-{\tilde{q}}]\sigma }<n$, we deduce by H\"older's inequality that
\begin{align*}
R^{-n} \int_{B_R(x)}V^{\sigma}dy&=R^{-n} \int_{B_R(x)}{\frac{{|u|}^{ [{2^{*}_{s}(\alpha)}-{\tilde{q}}]\sigma}}{|y|^  {  \sigma {\alpha}    }   }}dy=R^{-n} \int_{B_R(x)} \frac{1}{{|y|^  { t \sigma {\alpha}    }   }  } \cdot {\frac{{|u|}^{ [{2^{*}_{s}(\alpha)}-{\tilde{q}}]\sigma}}{|y|^  { (1-t) \sigma {\alpha}    }   }}dy  \\
&\leq R^{-n}  { \Big( \int_{B_R(0)} \frac{dy}{{|y|^  {  \frac{ t \sigma {\alpha}}{ 1- [{2^{*}_{s}(\alpha)}-{\tilde{q}}]\sigma }       }   }  }  \Big ) }^{1- [{2^{*}_{s}(\alpha)}-{\tilde{q}}]\sigma   }   { \Big( \int_{B_R(x)}  {\frac{{|u|}}{|y|^r  }}dy  \Big ) }^{[{2^{*}_{s}(\alpha)}-{\tilde{q}}]\sigma   }   \\
&\leq CR^{-t\alpha \sigma-n[{2^{*}_{s}(\alpha)}-{\tilde{q}}]\sigma}   { \Big( \int_{B_R(x)}  {\frac{{|u|}}{|y|^r }}dy  \Big ) }^{[{2^{*}_{s}(\alpha)}-{\tilde{q}}]\sigma   }
\end{align*}
where $t:=\frac{\tilde{q}}{2^{*}_{s}(\alpha)}$ and $r:={\frac{(1-t) {\alpha}}{{2^{*}_{s}(\alpha)}-{\tilde{q}}} }=\frac{\alpha}{2^{*}_{s}(\alpha)}$. Therefore,
\begin{align*}
&R^{s+\frac{n}{{\tilde{q}}}-\frac{n}{2}} { \Big(  R^{-n} \int_{ B_R(x) }V^{\sigma}dy    \Big) }^{ \frac{1}{{\tilde{q}} \sigma} }  \\
\leq & R^{s+\frac{n}{{\tilde{q}}}-\frac{n}{2}} \Big\{ CR^{-t\alpha \sigma-n[{2^{*}_{s}(\alpha)}-{\tilde{q}}]\sigma}   { \Big( \int_{B_R(x)}  {\frac{{|u|}}{|y|^r}}dy  \Big ) }^{[{2^{*}_{s}(\alpha)}-{\tilde{q}}]\sigma   }  \Big\}^{\frac{1}{{\tilde{q}} \sigma}}  \\
\leq & C \Big\{  R^{  (s+\frac{n-t\alpha}{{\tilde{q}}}-\frac{n}{2})
{\frac{{\tilde{q}}}{{2^{*}_{s}(\alpha)}-{\tilde{q}}}}   }   R^{-n}   {  \int_{B_R(x)}  {\frac{{|u|}}{|y|^r }}dy   }   \Big\}^{  \frac{{2^{*}_{s}(\alpha)}-{\tilde{q}}}{\tilde{q}}  } \\
= & C \Big\{  R^{   \frac{n-2s}{2}+r }   R^{-n}   {  \int_{B_R(x)}  {\frac{{|u|}}{|y|^  { r }   }}dy   }   \Big\}^{  \frac{{2^{*}_{s}(\alpha)}-{\tilde{q}}}{\tilde{q}}  }
\leq  C {||u||}^{\frac{{2^{*}_{s}(\alpha)}-{\tilde{q}}}
{\tilde{q}}}_{ L^{1,\frac{n-2s}{2}+r }(\R^{n},|y|^{-r})}:=C_{\sigma}.
\end{align*}
Since $u={\ell}_sg$, and by Lemma \ref{lemma3.4},
$$   \int_{ \R^{n} }  \frac{ |u(y)|^{ 2^*_{s,\alpha} } }  {  |y|^{\alpha} } dy =\int_{ \R^{n} }  {|{\ell}_sg(y)|}^{\tilde{q}}V(y) dy \leq (CC_{\sigma})^{\tilde{q}} ||g||^{\tilde{q}}_{L^2} \leq C ||u||_{{\dot{H}}^s(\R^{n})}^{\tilde{q}} {       ||u||^{^{{2^{*}_{s}(\alpha)}-{\tilde{q}}                      }}_{  L^{1,\frac{n-2s}{2}+r}(\R^{n},|y|^{-r}) } }.$$
Then, for any $\theta=\frac{{\tilde{q}}}{ 2^*_{s}(\alpha) }$ satisfying $\max \{ \frac{2}{2^*_{s}(\alpha)}, \frac{2^*_{s}-1}{2^*_{s}(\alpha)}  \} < \theta < 1$ and any $p\in[1,2^*_{s}(\alpha))$, we have
\begin{equation*}
 \Big( \int_{ \R^{n} }  \frac{ |u(y)|^{ 2^*_{s}(\alpha)} }  {  |y|^{\alpha} } dy  \Big)^{ \frac{1}{  2^*_{s} (\alpha)  }}  \leq C ||u||_{{\dot{H}}^s(\R^{n})}^{\theta} ||u||^{1-\theta}_{  L^{p,\frac{n-2s}{2}p+pr}(\R^{n},|y|^{-pr}) }.
\end{equation*} \qed  \\

\noindent\textbf{Proof of Corollary \ref{coro1.5}}

For $n\geq3$ and any $u \in {C}_{0}^{\infty}(\R^{n})$, we have
  $$u(x)={\Delta}^{-1}\Delta u=C_1\int_{\R^{n}}\frac{\Delta u(y)}{{|x-y|}^{n-2}}dy=C_2\int_{\R^{n}}\frac{(x-y)\nabla u(y)}{{|x-y|}^n}dy,$$
Thus
$$|u(x)| \leq |{C_2}|\int_{\R^{n}}\frac{| {\nabla u(y)} |}{{|x-y|}^{n-1}}dy \leq C {\ell}_{1}(| {\nabla u} |)(x)$$
where $C_1=C_1(n)$, $C_2=C_2(n)$ and $C=C(n)>0$ are different constants. These inequalities hold for $n=2$ via the logarithmic kernel(See \cite{GPAP}). By density of ${C}_{0}^{\infty}(\R^{n})$ in ${D}^{1,p}(\R^{n})$, it is also true for any $u \in {D}^{1,p}(\R^{n})(n\geq 2)$.

Take $\tilde{s}=1$, $\tilde{p}=p>1$, $\max \{ {p, p^*-1} \} < {\tilde{q}}< p^*(\alpha)$, $W(y)\equiv 1$, $V(y)=\frac{  {|u(y)|}^  { p^{*}(\alpha)-{\tilde{q}}}    }   {  |y|^{\alpha}  }$ and $\sigma=\frac{1}{p^{*}-\tilde{q}}>1$ in Lemma \ref{lemma3.4}. The remain argument is similar to the case in $ {\dot{H}}^s(\R^{n})$.   \qed

\begin{lemma} (Theorem 1 in \cite{GPAP}) \label{lemma3.6}
Let $s \in (0,1)$, $n>2s$ and $ 2^*_{s}=\frac{2n}{n-2s} $. Then there exists $C=C(n,s)>0$ such that for any $ \max\{ \frac{2}{2^*_{s}},1-\frac{1}{2^*_{s}} \} < \theta <1 $ and for any $1\leq p <2^*_{s}$
\begin{equation}\label{eq3.6}
 ||u||_{ L^{ 2^*_{s} }(\R^{n})} \leq C ||u||_{{\dot{H}}^s(\R^{n})}^{\theta} ||u||^{1-\theta}_{  L^{p,\frac{n-2s}{2}p}(\R^{n}) },~~~~ \forall u \in {\dot{H}}^s(\R^{n}).
\end{equation}
\end{lemma}

\begin{remark}
If $\alpha=0$ in Proposition \ref{pro1.4}, then inequality (\ref{eq1.6}) becomes inequality (\ref{eq3.6}).
\end{remark}

\section{Solving the minimization problems (\ref{eq1.8})-(\ref{eq1.9})}
In this section, we solve the minimization problems (\ref{eq1.8})-(\ref{eq1.9}). Using the embeddings (\ref{eq1.06}) and the inequality (\ref{eq1.6}), we can prove the existence of minimizers for
\begin{equation*}
   S_{\mu}(n,s,\gamma,\alpha)=\mathop {\inf }\limits_{u \in \dot{H}^{s}(\R^{n})) \setminus \{0\}  }   \frac{ ||u||^2 }
   {  {B_{\alpha}(u,u)}^{\frac{1}{  2^{\#}_{\mu}(\alpha)  }}  }
\end{equation*}
and
\begin{equation*}
   \Lambda(n,s,\gamma,\alpha)=  \mathop {\inf }\limits_{u \in \dot{H}^{s}(\R^{n}) \setminus \{0\}}   \frac{||u||^2}{\Big( \int_{\R^{n}} \frac{{|u|}^{ { 2^{*}_{s} }(\alpha)}}{|x|^{\alpha}}dx  \Big)^{\frac{2}{  2^{*}_{s}(\alpha)  }}}
\end{equation*}
where $B_{\alpha}(\cdot,\cdot)$ was defined in (\ref{eq2.01}). We can derive the following results:
\begin{proposition}\label{pro1.7}
Let $s \in(0,1)$. Then  \\
$(1)$ If $0<\alpha< 2s<n$, $\mu \in (0,n)$ and $\gamma<\gamma_{H}$, then $S_{\mu}(n,s,\gamma,\alpha)$ is attained in $\dot{H}^{s}(\R^n);$ \\
$(2)$ If $n>2s$, $\mu \in (0,n)$ and $0 \leq \gamma < \gamma_{H}$, then $S_{\mu}(n,s,\gamma,0)$ is  attained in $\dot{H}^{s}(\R^n);$ \\
$(3)$ If $0<\alpha< 2s<n$ and $\gamma<\gamma_{H}$, then $\Lambda(n,s,\gamma,\alpha)$ is attained in $\dot{H}^{s}(\R^n);$ \\
$(4)$ If $n>2s$ and $0 \leq \gamma < \gamma_{H}$, then $\Lambda(n,s,\gamma,0)$ is  attained in $\dot{H}^{s}(\R^n).$
\end{proposition}
\begin{remark}
 We only prove (1)-(2) in this Section since the strategy can be applied to prove (3)-(4); Although (3) has been proved in \cite{NGSS}, our method is more direct and effective; We can derive $S_{\mu}(n,s,\gamma,\alpha)\geq \frac{\Lambda(n,s,\gamma,\alpha)}{ {C(n,\mu)}^{\frac{1}{  2^{\#}_{\mu}(\alpha)  }} }$ and $S_{\mu}(n,s,0,0)= \frac{\Lambda(n,s,0,0)}{ {C(n,\mu)}^{\frac{1}{  2^{\#}_{\mu} }} }$ from (\ref{eq2.6}).
\end{remark}

\noindent\textbf{Proof of Proposition \ref{pro1.7} }

$\textbf{(1)}$ If $ 0<\alpha<2s<n $ and $\gamma <\gamma_H$, let $\{u_k\}$ be a minimizing sequence of $S_{\mu}(n,s,\gamma,\alpha)$, that is
$$  {B_{\alpha}(u_k,u_k)}=1,~~~~~~~~||u_k||^2 \rightarrow S_{\mu}(n,s,\gamma,\alpha).$$
Then the embeddings (\ref{eq1.06}), the improved Sobolev inequality (\ref{eq1.6}) and (\ref{eq2.6}) imply that there exists $C>0$ such that
$$   0<C \leq ||u_k||_{  L^{2,{n-2s}+2r}(\R^{n},|y|^{-2r}) } \leq C^{-1}$$
where $r=\frac{\alpha}{ 2^*_{s}(\alpha) }$. For any $k\geq1$, we may find ${\lambda}_k>0$ and $x_k \in \R^{n}$ such that
$$  {\lambda}_k^{-2s+2r} \int_{B_{{\lambda}_k}(x_k)} \frac{|u_k(y)|^2}{  |y|^{2r} }dy > ||u_k||^2_{  L^{2,{n-2s}+2r}(\R^{n},|y|^{-2r}) } -\frac{C}{2k} \geq C_1>0.$$
Let $v_k(x)={\lambda}_k^{ \frac{n-2s}{2} }u_k({\lambda}_kx)$ and ${\tilde{x}}_k=\frac{x_k}{{\lambda}_k}$, then
\begin{equation}\label{eq4.4}
\int_{B_{1}({\tilde{x}}_k)} \frac{|v_k(x)|^2}{  |x|^{2r} }dx  \geq C_1>0.
\end{equation}
Since $S_{\mu}(n,s,\gamma,\alpha)$ is invariant under the previous dilation given by $\lambda_k$, we have
$$  {B_{\alpha}(v_k,v_k)}=1,~~~~~~~~||v_k||^2 \rightarrow S_{\mu}(n,s,\gamma,\alpha).$$
By H$\ddot{o}$lder's inequality,
\begin{align*}
   0<C_1 \leq \int_{B_{1}({\tilde{x}}_k)} \frac{|v_k(x)|^2}{  |x|^{2r} }dx  &\leq  { \Big( \int_{B_{1}({\tilde{x}}_k)}dx\Big)}^{1-\frac{2}{2^*_{s}(\alpha)}}
 { \Big(  \int_{B_{1}({\tilde{x}}_k)} \frac{|v_k(x)|^{2^*_{s}(\alpha)}}{  |x|^{\alpha} }dx  \Big)}^{\frac{2}{2^*_{s}(\alpha)}}  \\
   &\leq C { \Big(  \int_{B_{1}({\tilde{x}}_k)} \frac{|v_k(x)|^{2^*_{s}(\alpha)}}{  |x|^{\alpha} }dx  \Big)}^{\frac{2}{2^*_{s}(\alpha)}}.
\end{align*}
Therefore,
\begin{equation}\label{eq4.5}
 \int_{B_{1}({\tilde{x}}_k)} \frac{|v_k(x)|^{2^*_{s}(\alpha) }}{  |x|^{\alpha} }dx  \geq C>0.
\end{equation}
We claim that $\{  {\tilde{x}}_k \}$ is bounded. Indeed, if on the contrary, $|{\tilde{x}}_k| \rightarrow  +\infty$, then for any $x \in B_{1}({\tilde{x}}_k)$, $|x| \geq |{\tilde{x}}_k|-1$ for $k$ large. Therefore,
\begin{align*}
\int_{B_{1}({\tilde{x}}_k)} \frac{|v_k(x)|^{2^*_{s}(\alpha)}}{  |x|^{\alpha} }dx  &\leq    \frac{1}{  (|{\tilde{x}}_k|-1)^{\alpha} }   \int_{B_{1}({\tilde{x}}_k)}   |v_k(x)|^{2^*_{s}(\alpha)}    dx  \\
&\leq       \frac{C}{  (|{\tilde{x}}_k|-1)^{\alpha} }   { \Big( \int_{B_{1}({\tilde{x}}_k)}   |v_k(x)|^{2^*_{s}}    dx  \Big) }^{\frac{n-\alpha}{n}} \\
&\leq       \frac{C}{  (|{\tilde{x}}_k|-1)^{\alpha} } ||v_k||^{\frac{2(n-\alpha)}{n-2s}}_{{\dot{H}}^s(\R^{n})}  \leq       \frac{\tilde{C}}{  (|{\tilde{x}}_k|-1)^{\alpha} }\rightarrow 0.
\end{align*}
as $k \rightarrow  +\infty$, which contradicts to (\ref{eq4.5}). Hence, $\{  {\tilde{x}}_k \}$ is bounded, from (\ref{eq4.4}) we may find $R>0$ such that
\begin{equation}\label{eq4.6}
  \int_{B_{R}(0)} \frac{|v_k(x)|^2}{  |x|^{2r} }dx  \geq C_1>0.
\end{equation}
Since $||{v}_k||=||{u}_k||\leq C$, there exists a $v\in {\dot{H}}^s(\R^{n})$ such that
\begin{align} \label{eq4.7}
  v_k \rightharpoonup v ~~\mbox{in} ~~\dot{H}^{s}(\R^{n}),~~~~~~ v_k \rightarrow v ~~\mbox{a.e. ~~~~on}~~ \R^n
\end{align}
up to subsequences. According to Lemma \ref{lemma2.3}, we have $\frac{v_k}{ {|x|}^{r}  }    \rightarrow   \frac{v}{ {|x|}^{r}  } ~~\mbox{in}~~L^2_{loc}(\R^{n})$ since $r=\frac{\alpha}{ 2^*_{s}(\alpha) }<s$, therefore
$$ \int_{B_{R}(0)} \frac{|v(x)|^2}{  |x|^{2r} }dx  \geq C_1>0, $$
and we deduce that $v\not \equiv 0$. We may verify as Lemma \ref{lemma2.8} that
\begin{align*}
1={B_{\alpha}(v_k,v_k)}={B_{\alpha}(v_k-v,v_k-v)}
+{B_{\alpha}(v,v)}+o(1).
\end{align*}
By the weak convergence $v_k \rightharpoonup v$ in ${\dot{H}}^s(\R^{n})$,
\begin{align*}
     S_{\mu}(n,s,\gamma,\alpha)&=\mathop {\lim }\limits_{k  \to \infty} ||v_k||^2=  ||v||^2+ \mathop {\lim }\limits_{k  \to \infty} ||v_k-v||^2   \\
     &\geq S_{\mu}(n,s,\gamma,\alpha)  \Big({B_{\alpha}(v,v)} \Big)^{   \frac{1}  {  { 2^{\#}_{\mu} }(\alpha) }  }    \\
     &+S_{\mu}(n,s,\gamma,\alpha)  \Big( \mathop {\lim }\limits_{k  \to \infty} {B_{\alpha}(v_k-v,v_k-v)} \Big)^{ \frac{1}{  2^{\#}_{\mu}(\alpha)   }} \\
     &\geq S_{\mu}(n,s,\gamma,\alpha)  \Big({B_{\alpha}(v,v)}
     + \mathop {\lim }\limits_{k  \to \infty} {B_{\alpha}(v_k-v,v_k-v)} \Big)^{ \frac{1}{  2^{\#}_{\mu}(\alpha)   }}\\
&=S_{\mu}(n,s,\gamma,\alpha).
\end{align*}
Here we use the fact that $(a+b)^{\frac{1}{  2^{\#}_{\mu}(\alpha)   }} \leq a^{\frac{1}{  2^{\#}_{\mu}(\alpha)   }}+b^{\frac{1}{  2^{\#}_{\mu}(\alpha)   }}$, $\forall a\geq0, b\geq0$ and $2^{\#}_{\mu}(\alpha)>1.$\\
So we have
$${B_{\alpha}(v,v)}=1,~~~~~~\mathop {\lim }\limits_{k  \to \infty}{B_{\alpha}(v_k-v,v_k-v)}=0,$$
since $v \not \equiv 0$. It results
$$ S_{\mu}(n,s,\gamma,\alpha)=||v||^2,~~~~\mathop {\lim }\limits_{k  \to \infty} ||v_k-v||^2=0.$$
By formula (A.11) in \cite{RSER},
$$ \int_{\R^{n}} |(-\Delta)^{\frac{s}{2} }|v||^2 dx \leq  \int_{\R^{n}} |(-\Delta)^{\frac{s}{2} }v|^2 dx ,  $$
Hence, $|v|$ is also a minimizer of $S_{\mu}(n,s,\gamma,\alpha)$, we can assume $v \geq 0$. Thus $S_{\mu}(n,s,\gamma,\alpha)$ is achieved if $ 0<\alpha<2s $ and $\gamma <\gamma_H .$

$\textbf{(2)}$ If $\alpha=0$ and $0\leq \gamma <\gamma_H$, we are inspired by the method introduced by R. Filippucci in \cite{RFPP} and S. Dipierro in \cite{SDLM}. Let $\{u_k\}$ be a minimizing sequence of $S_{\mu}(n,s,\gamma,0)$, that is
$$
{B_{0}(u_k,u_k)}=1,~~~~~~~~S_{\mu}(n,s,\gamma,0)\leq ||u_k||^2<S_{\mu}(n,s,\gamma,0)+\frac{1}{k}.$$
From the fractional Polya-Szeg$\ddot{o}$ inequality in \cite{YJPK} and formula (A.11) in \cite{RSER}, we have
$$ \int_{\R^{n}} |(-\Delta)^{\frac{s}{2} }{|u_k|}^*|^2 dx \leq  \int_{\R^{n}} |(-\Delta)^{\frac{s}{2} }|u_k||^2 dx \leq  \int_{\R^{n}} |(-\Delta)^{\frac{s}{2} }u_k|^2 dx $$
where ${|u_k|}^*$ is the symmetric decreasing rearrangement of $|u_k|$.
Furthermore, it is clear(Theorem 3.4 in \cite{ELMA}) that
$$1={B_{0}(|u_k|,|u_k|)} \leq {B_{0}({|u_k|}^*,{|u_k|}^*)},~~~~~~~~\int_{\R^{n}} {\frac{|u_k|^2}{|x|^{2s}}}dx \leq \int_{\R^{n}} {\frac{|{|u_k|}^*|^2}{|x|^{2s}}}dx.$$
Denote $v_k:={|u_k|}^*$, then $v_k$ is radial symmetric and decreasing. Since $0\leq \gamma <\gamma_H$, we have
$$
S_{\mu}(n,s,\gamma,0) \leq \frac{ ||v_k||^2 }
   {  {B_{0}(v_k,v_k)}^{\frac{1}{  2^{\#}_{\mu}  }}}\leq ||v_k||^2 \leq ||u_k||^2<S_{\mu}(n,s,\gamma,0)+\frac{1}{k}.$$
Therefore, $\{v_k\}$ is a minimizing sequence of $S_{\mu}(n,s,\gamma,0)$ and $||v_k||$ is uniformly bounded.\\
Noticing that ${B_{0}(v_k,v_k)} \geq 1$, the embeddings ${\dot{H}}^s(\R^{n}) \hookrightarrow L^{2^*_{s}}(\R^{n}) \hookrightarrow L^{2,n-2s}(\R^{n})$(See Section 3),  inequality (\ref{eq2.6}) and Lemma \ref{lemma3.6} imply that there exists $C>0$ such that
$$   0<C \leq ||v_k||_{  L^{2,{n-2s}}(\R^{n}) } \leq C^{-1} .$$
Therfore we may find ${\lambda}_k>0$ and $x_k \in \R^{n}$ such that
$$  {\lambda}_k^{-2s} \int_{B_{{\lambda}_k}(x_k)} |v_k(y)|^2 dy > ||v_k||^2_{  L^{2,{n-2s}}(\R^{n}) } -\frac{C}{2k} \geq C_1>0.$$
Let $\tilde{v}_k(x)={\lambda}_k^{ \frac{n-2s}{2} }v_k({\lambda}_kx)$ and ${\tilde{x}}_k=\frac{x_k}{{\lambda}_k}$, we see that $\{\tilde{v}_k\}$ is also a minimizing sequence of $S_{\mu}(n,s,\gamma,0)$ and satisfies
\begin{equation}\label{eq4.9}
\int_{B_{1}({\tilde{x}}_k)} |\tilde{v}_k(x)|^2dx  \geq C_1>0.
\end{equation}
Since $||\tilde{v}_k||=||{v}_k||\leq C$, there exists $\tilde{v} \in {\dot{H}}^s(\R^{n})$ such that $\tilde{v}_k \rightharpoonup \tilde{v} ~~\mbox{in} ~~\dot{H}^{s}(\R^{n})$ up to subsequences, we need to prove $\tilde{v}\not \equiv 0.$

Case(1): If ${\tilde{x}}_k $ is unbounded, we assume that $|{\tilde{x}}_k| \to +\infty$ up to subsequence. Since the sequence $\{ \tilde{v}_k(x) \}$ is radial symmetric and decreasing, from (\ref{eq4.9}), we have for all $k$ that
\begin{equation*}\label{eq4.10}
\int_{B_{2}(0)} |\tilde{v}_k(x)|^2dx  \geq \int_{B_{1}(0)} |\tilde{v}_k(x+{\tilde{x}}_k)|^2dx=\int_{B_{1}({\tilde{x}}_k)} |\tilde{v}_k(x)|^2dx  \geq  C_1>0.
\end{equation*}
Since ${\dot{H}}^s(\R^{n}) \hookrightarrow L^{2}_{loc}(\R^{n})$ is compact(see Corollary 7.2 of \cite{ENEV}), we have
\begin{equation*}\label{eq4.11}
\int_{B_{2}(0)} |\tilde{v}(x)|^2dx  \geq  C_1>0.
\end{equation*}

Case(2): If ${\tilde{x}}_k $ is bounded, from (\ref{eq4.9}) we may find $R>0$ such that
\begin{equation*}\label{eq4.12}
  \int_{B_{R}(0)}|\tilde{v}_k(x)|^2dx  \geq C_1>0
\end{equation*}
and we also derive
\begin{equation*}\label{eq4.13}
\int_{B_{R}(0)} |\tilde{v}(x)|^2dx  \geq  C_1>0.
\end{equation*}
Thus we have $\tilde{v}\not \equiv 0$. The rest is the same as the proof of Proposition \ref{pro1.7}-(1), then Proposition \ref{pro1.7}-(2) holds.

$\textbf{(3)}$ The proof is similar to Proposition \ref{pro1.7}-(1). Although Proposition \ref{pro1.7}-(3) has been proved in \cite{NGSS}, the strategy we adopted in Proposition \ref{pro1.7}-(1) is more direct and effective.

$\textbf{(4)}$ Imitate the proof of Proposition \ref{pro1.7}-(2).

\qed

\begin{remark}
To prove Proposition \ref{pro1.7}-(2), firstly we choose a minimizing sequence $\{u_k\}$ of $S_{\mu}(n,s,\gamma,0)$, then we prove $v_k={|u_k|}^*$ is also a minimizing sequence of $S_{\mu}(n,s,\gamma,0)$ since $0\leq \gamma <\gamma_H$. Since $v_k$ is radial symmetric and decreasing, we can easily eliminate vanishing. If $\alpha>0$ and $0\leq \gamma <\gamma_H$, the same strategy can be applied to the proof of Proposition \ref{pro1.7}-(1). When it comes to $\alpha>0$ and $\gamma <0$, we fail to prove that $v_k={|u_k|}^*$ is a minimizing sequence of $S_{\mu}(n,s,\gamma,\alpha)$, but (\ref{eq1.06}) and (\ref{eq1.6}) are very effective in this situation.
\end{remark}

\section{proof of  Theorem \ref{th1.1}}
We shall now use the minimizers of $S_{\mu}(n,s,\gamma,\alpha)$ and $\Lambda(n,s,\gamma,\beta)$ obtained in Proposition \ref{pro1.7}, to prove the existence of a nontrivial weak solution for equation (\ref{eq1.1}). Recall that, the energy functional associated to (\ref{eq1.1}) is:
\begin{equation}\label{eq5.1}
  I(u)=\frac{1}{2}||u||^2-\frac{1}{2^{*}_{s}(\beta)}\int_{\R^n}{\frac{{|u|}^{ {2^{*}_{s}}(\beta)}}{|x|^{\beta}}}dx-\frac{1}{2 \cdot { 2^{\#}_{\mu} }(\alpha) } B_{\alpha}(u,u),~~~~\forall u \in \dot{H}^{s}(\R^n)
\end{equation}
where $B_{\alpha}(\cdot,\cdot)$ was defined in (\ref{eq2.01}). Fractional Sobolev and Hardy-Sobolev inequalities yield that $I \in C^1(\dot{H}^{s}(\R^n),\R)$ such that
\begin{align*}
{\langle I'(u),\phi\rangle}&={\langle u,\phi\rangle}
   -\int_{\R^n}{\frac{ {|u|}^{ {2^{*}_{s}}(\beta)-2}u\phi }{|x|^{\beta}}}dx
   -\int_{\R^{n}} \big [   I_{\mu}* F_{\alpha}(\cdot,u)  \big](x)f_{\alpha}(x,u)\phi(x)dx.
\end{align*}
Note that a nontrivial critical point of $I$ is a nontrivial weak solution to equation (\ref{eq1.1}).

\begin{lemma}(Mountain pass lemma, \cite{AMPH})  \label{lemma5.1}
Let $(E,||\cdot||)$ be a Banach space and $I\in  C^1(E,\R) $ satisfying the following conditions:  \\
$(1)$  $I(0)=0 ,$ \\
$(2)$ There exist $\rho,r>0 $ such that $I(u)\geq \rho $ for all $u \in E$ with $||u||=r,$  \\
$(3)$ There exist $v_0 \in E$ such that $  \lim_{t \to +\infty} {\sup }I(tv_0)<0.$  \\
Let $t_0>0$ be such that $||t_0v_0||>r$ and $I(t_0v_0)<0$, and define
$$  c:=\mathop {\inf }\limits_{g\in \Gamma}   \mathop {\sup }\limits_{t\in [0,1]} I(g(t)),$$
where
$$\Gamma:=\Big \{  g \in C^0([0,1],E) :g(0)=0,g(1)=t_0v_0   \Big \}.$$
Then, $c\geq \rho >0$ and there exists a $(PS)$ sequence $\{u_k\} \subset E$ for $I$ at level $c$, i.e.
$$\lim_{k \to +\infty}I(u_k)=c~~\mbox{and}~~ \lim_{k \to +\infty} I'(u_k)=0~~\mbox{strongly in}~~E'.$$\\
\end{lemma}

We now use Lemma \ref{lemma5.1} to prove the following Propositions.

\begin{proposition}  \label{prop5.2}
Let $s \in(0,1)$, $0<\alpha, \beta<2s< n $, $\mu \in (0,n)$ and $\gamma<\gamma_{H}$. Consider the functional $I$ defined in $(\ref{eq5.1})$ on the Banach space $\dot{H}^{s}(\R^n)$. Then there exists a $(PS)$ sequence $\{u_k\} \subset \dot{H}^{s}(\R^n)$ for $I$ at some $c\in (0,c^*)$, i.e.
\begin{equation}\label{eq5.2}
\lim_{k \to +\infty}I(u_k)=c~~\mbox{and}~~ \lim_{k \to +\infty} I'(u_k)=0~~\mbox{strongly in}~~\dot{H}^{s}(\R^n)'
\end{equation}
where $$c^*:=\min \Big \{ \frac{2^{\#}_{\mu}(\alpha)-1}{2 \cdot 2^{\#}_{\mu}(\alpha)} {S_{\mu}(n,s,\gamma,\alpha)}^{\frac{{ 2^{\#}_{\mu} }(\alpha)}{ { 2^{\#}_{\mu} }(\alpha)-1 }}, \frac{2s-\beta}{2(n-\beta)} \Lambda(n,s,\gamma,\beta)^{\frac{n-\beta}{2s-\beta}} \Big \}.$$

\end{proposition}

\begin{proof}
We now verify the conditions of Lemma \ref{lemma5.1}.
For any $u\in \dot{H}^{s}(\R^{n})$,
\begin{align*}
  I(u)&=\frac{1}{2}||u||^2-\frac{1}{2^{*}_{s}(\beta)}\int_{\R^n}{\frac{{|u|}^{ {2^{*}_{s}}(\beta)}}{|x|^{\beta}}}dx-\frac{1}{2 \cdot { 2^{\#}_{\mu} }(\alpha) } B_{\alpha}(u,u) \\
  &\geq\frac{1}{2}||u||^2-C_1||u||^{2^{*}_{s}(\beta)}-C_2||u||^{  2 \cdot 2^{\#}_{\mu} (\alpha) }.
\end{align*}
Since $s \in(0,1)$, $0<\alpha, \beta<2s< n $ and $\mu \in (0,n)$, we have that $2^{*}_{s}(\beta)>2$ and $2 \cdot 2^{\#}_{\mu} (\alpha)>2^{*}_{s}(\alpha)>2$. Therefore, there exists $r>0$ small enough such that
$$  \mathop {\inf }\limits_{||u||=r} I(u)>0=I(0),     $$
so $(1)$ and $(2)$ of Lemma \ref{lemma5.1} are satisfied.\\
From
\begin{align*}
I(tu)=\frac{t^2}{2}||u||^2-\frac{t^ {2^{*}_{s}(\beta)}}{2^{*}_{s}(\beta)}\int_{\R^n}{\frac{{|u|}^{ {2^{*}_{s}}(\beta)}}{|x|^{\beta}}}dx-\frac{ t^{2 \cdot { 2^{\#}_{\mu} }(\alpha) } }{2 \cdot { 2^{\#}_{\mu} }(\alpha) } B_{\alpha}(u,u),
\end{align*}
we derive that $\lim_{t \to +\infty} I(tu)=-\infty$ for any $u\in \dot{H}^{s}(\R^{n})$. Consequently, for any fixed $v_0\in \dot{H}^{s}(\R^{n})$, there exists ${t_{v_0}}>0$ such that $||t_{v_0} v_0||>r$ and $I(t_{v_0} v_0)<0$. So $(3)$ of Lemma \ref{lemma5.1} is satisfied.

Using (1) and (3) in Proposition \ref{pro1.7}, we obtain a minimizer $U_{\gamma,\alpha} \in \dot{H}^{s}(\R^n)$ for $S_{\mu}(n,s,\gamma,\alpha)$ and $V_{\gamma,\beta} \in \dot{H}^{s}(\R^n)$ for $\Lambda(n,s,\gamma,\beta)$ respectively. So there exist
$$
v_0:=\left \{ \begin{array}{ll} U_{\gamma,\alpha},~~~~\mbox{if}&  \frac{2^{\#}_{\mu}(\alpha)-1}{2 \cdot 2^{\#}_{\mu}(\alpha)} {S_{\mu}(n,s,\gamma,\alpha)}^{\frac{{ 2^{\#}_{\mu} }(\alpha)}{ { 2^{\#}_{\mu} }(\alpha)-1 }} \leq \frac{2s-\beta}{2(n-\beta)} \Lambda(n,s,\gamma,\beta)^{\frac{n-\beta}{2s-\beta}}; \\ V_{\gamma,\beta} ,~~~~\mbox{if}&  \frac{2^{\#}_{\mu}(\alpha)-1}{2 \cdot 2^{\#}_{\mu}(\alpha)} {S_{\mu}(n,s,\gamma,\alpha)}^{\frac{{ 2^{\#}_{\mu} }(\alpha)}{ { 2^{\#}_{\mu} }(\alpha)-1 }} > \frac{2s-\beta}{2(n-\beta)} \Lambda(n,s,\gamma,\beta)^{\frac{n-\beta}{2s-\beta}}
\end{array}    \right .   $$
and $t_0>0$ such that $||t_0v_0||>r$ and $I(t_0v_0)<0$. We can define
$$  c:=\mathop {\inf }\limits_{g\in \Gamma}   \mathop {\sup }\limits_{t\in [0,1]} I(g(t))$$
where
$$\Gamma:=\Big \{  g \in C^0([0,1],\dot{H}^{s}(\R^{n})) :g(0)=0, g(1)=t_0v_0   \Big \}.$$
Clearly we have $c>0$. For the case of $v_0=U_{\gamma,\alpha}$, we can derive that $$0<c<\frac{2^{\#}_{\mu}(\alpha)-1}{2 \cdot 2^{\#}_{\mu}(\alpha)} {S_{\mu}(n,s,\gamma,\alpha)}^{\frac{{ 2^{\#}_{\mu} }(\alpha)}{ { 2^{\#}_{\mu} }(\alpha)-1 }}.$$
In fact, $\forall t \geq 0$, we have
\begin{align*}
I(tU_{\gamma,\alpha}) \leq f_1(t):= \frac{t^2}{2}||U_{\gamma,\alpha}||^2-\frac{ t^{2 \cdot { 2^{\#}_{\mu} }(\alpha) } }{2 \cdot { 2^{\#}_{\mu} }(\alpha) } B_{\alpha}(U_{\gamma,\alpha},U_{\gamma,\alpha}).
\end{align*}
Straightforward computations yield that $f_1(t)$ attains its maximum at the point
 $$\tilde{t}=\Big(  \frac{||U_{\gamma,\alpha}||^2}{  B_{\alpha}(U_{\gamma,\alpha},U_{\gamma,\alpha}) }  \Big)^{\frac{1}{ 2[{ 2^{\#}_{\mu} }(\alpha)-1] }}$$
and
$$ \mathop {\sup }\limits_{t \geq 0} f_1(t)=\frac{2^{\#}_{\mu}(\alpha)-1}{2 \cdot 2^{\#}_{\mu}(\alpha)} \Big(  \frac{||U_{\gamma,\alpha}||^2}{  {B_{\alpha}(U_{\gamma,\alpha},U_{\gamma,\alpha})}^{ \frac{1}{2^{\#}_{\mu}(\alpha)} } }  \Big)^{\frac{{ 2^{\#}_{\mu} }(\alpha)}{ { 2^{\#}_{\mu} }(\alpha)-1 }}=\frac{2^{\#}_{\mu}(\alpha)-1}{2 \cdot 2^{\#}_{\mu}(\alpha)} {S_{\mu}(n,s,\gamma,\alpha)}^{\frac{{ 2^{\#}_{\mu} }(\alpha)}{ { 2^{\#}_{\mu} }(\alpha)-1 }}.$$
We obtain that,
\begin{equation} \label{eq5.4}
 \mathop {\sup }\limits_{t \geq 0} I({tU_{\gamma,\alpha}}) \leq  \mathop {\sup }\limits_{t \geq 0} f_1(t)=\frac{2^{\#}_{\mu}(\alpha)-1}{2 \cdot 2^{\#}_{\mu}(\alpha)} {S_{\mu}(n,s,\gamma,\alpha)}^{\frac{{ 2^{\#}_{\mu} }(\alpha)}{ { 2^{\#}_{\mu} }(\alpha)-1 }}.
\end{equation}
The equality does not hold in (\ref{eq5.4}), otherwise, we would have that $\mathop {\sup }\limits_{t \geq 0} I({tU_{\gamma,\alpha}}) =  \mathop {\sup }\limits_{t \geq 0} f_1(t)$. Let $t_1>0$ where $\mathop {\sup }\limits_{t \geq 0} I(tU_{\gamma,\alpha})$ is attained. We have
$$ f_1(t_1)-\frac{t_1^ {2^{*}_{s}(\beta)}}{2^{*}_{s}(\beta)}\int_{\R^n}{\frac{{|U_{\gamma,\alpha}|}^{ {2^{*}_{s}}(\beta)}}{|x|^{\beta}}}dx=f_1(\tilde{t})$$
which means that $f_1(t_1)>f_1(\tilde{t})$ since $t_1>0$. This contradicts the fact that $\tilde{t}$ is the unique maximum point of $f_1(t)$.
Thus
\begin{equation} \label{eq5.5}
 \mathop {\sup }\limits_{t \geq 0} I({tU_{\gamma,\alpha}})<  \mathop {\sup }\limits_{t \geq 0} f_1(t)=\frac{2^{\#}_{\mu}(\alpha)-1}{2 \cdot 2^{\#}_{\mu}(\alpha)} {S_{\mu}(n,s,\gamma,\alpha)}^{\frac{{ 2^{\#}_{\mu} }(\alpha)}{ { 2^{\#}_{\mu} }(\alpha)-1 }}.
\end{equation}

For the case of $v_0=V_{\gamma,\beta}$, similarly, we can verify
\begin{equation} \label{eq5.6}
 \mathop {\sup }\limits_{t \geq 0} I({tV_{\gamma,\beta}}) <\frac{2s-\beta}{2(n-\beta)} \Lambda(n,s,\gamma,\beta)^{\frac{n-\beta}{2s-\beta}}
\end{equation}
and thus $0<c<\frac{2s-\beta}{2(n-\beta)} \Lambda(n,s,\gamma,\beta)^{\frac{n-\beta}{2s-\beta}}$.

From (\ref{eq5.5}) and (\ref{eq5.6}), we have
$$0<c<c^*:=\min \Big \{ \frac{2^{\#}_{\mu}(\alpha)-1}{2 \cdot 2^{\#}_{\mu}(\alpha)} {S_{\mu}(n,s,\gamma,\alpha)}^{\frac{{ 2^{\#}_{\mu} }(\alpha)}{ { 2^{\#}_{\mu} }(\alpha)-1 }}, \frac{2s-\beta}{2(n-\beta)} \Lambda(n,s,\gamma,\beta)^{\frac{n-\beta}{2s-\beta}} \Big \}.$$
Since (1)-(3) of Lemma \ref{lemma5.1} are satisfied, there exists a sequence $\{u_k\} \subset \dot{H}^{s}(\R^n)$ such that
$$\lim_{k \to +\infty}I(u_k)=c~~\mbox{and}~~ \lim_{k \to +\infty} I'(u_k)=0~~\mbox{strongly in}~~\dot{H}^{s}(\R^n)'.$$  \\
\end{proof}

\begin{proposition}  \label{prop5.02}
Let $s \in(0,1)$, $n>2s$, $\alpha=0<\beta<2s$ or $\beta=0<\alpha<2s$, $\mu \in (0,n)$ and $0\leq \gamma<\gamma_{H}$. Consider the functional $I$ defined in $(\ref{eq5.1})$ on the Banach space $\dot{H}^{s}(\R^n)$. Then there exists a $(PS)$ sequence $\{u_k\} \subset \dot{H}^{s}(\R^n)$ for $I$ at some $c\in (0,c^*)$, i.e.
\begin{equation*}
\lim_{k \to +\infty}I(u_k)=c~~\mbox{and}~~ \lim_{k \to +\infty} I'(u_k)=0~~\mbox{strongly in}~~\dot{H}^{s}(\R^n)'
\end{equation*}
where $$c^*:=\min \Big \{ \frac{2^{\#}_{\mu}(\alpha)-1}{2 \cdot 2^{\#}_{\mu}(\alpha)} {S_{\mu}(n,s,\gamma,\alpha)}^{\frac{{ 2^{\#}_{\mu} }(\alpha)}{ { 2^{\#}_{\mu} }(\alpha)-1 }}, \frac{2s-\beta}{2(n-\beta)} \Lambda(n,s,\gamma,\beta)^{\frac{n-\beta}{2s-\beta}} \Big \}.$$
\end{proposition}

\begin{proof}
Imitate the proof of Proposition \ref{prop5.2}. Since $0\leq \gamma<\gamma_{H}$, using (2) and (4) in Proposition \ref{pro1.7}, we obtain a minimizer $U_{\gamma} \in \dot{H}^{s}(\R^n)$ for $S_{\mu}(n,s,\gamma,0)$ and $V_{\gamma} \in \dot{H}^{s}(\R^n)$ for $\Lambda(n,s,\gamma,0)$ respectively. The rest is standard.
\end{proof}

\noindent\textbf{Proof of Theorem \ref{th1.1}}

\noindent \textbf{(I)} The case $s \in(0,1)$, $0<\alpha,\beta<2s<n$, $\mu \in (0,n)$ and $\gamma<\gamma_{H}$.

Let $\{u_k\}_{k\in \N}$ be a $(PS)$ sequence as in Proposition \ref{prop5.2}, i.e.
$$I(u_k) \rightarrow c  ,~~ I'(u_k) \rightarrow 0 ~~\mbox{strongly in}~~\dot{H}^{s}(\R^n)' ~~\mbox{as}~~ k \rightarrow +\infty.$$
Then
\begin{equation}\label{eq5.7}
  I(u_k)=\frac{1}{2}||u_k||^2-\frac{1}{2^{*}_{s}(\beta)}
  \int_{\R^n}{\frac{{|u_k|}^{ {2^{*}_{s}}(\beta)}}{|x|^{\beta}}}dx-\frac{1}{2 \cdot { 2^{\#}_{\mu} }(\alpha) } B_{\alpha}(u_k,u_k)=c+o(1)
\end{equation}
and
\begin{equation}\label{eq5.8}
  \langle I'(u_k),u_k\rangle=||u_k||^2-\int_{\R^n}{\frac{{|u_k|}^{ {2^{*}_{s}}(\beta)}}{|x|^{\beta}}}dx-B_{\alpha}(u_k,u_k)=o(1).
\end{equation}

From (\ref{eq5.7}) and (\ref{eq5.8}), if $2 \cdot { 2^{\#}_{\mu} }(\alpha) \geq {2^{*}_{s}(\beta)}>2$, we have
\begin{align*}
c+o(1)||u_k||&=I(u_k)-\frac{1}{2^{*}_{s}(\beta)}  \langle I'(u_k),u_k\rangle \geq  \Big(\frac{1}{2}-\frac{1}{2^{*}_{s}(\beta)} \Big)||u_k||^2.
\end{align*}
If ${ 2^{*}_{s}(\beta)}> 2 \cdot { 2^{\#}_{\mu} }(\alpha)>2$, we have
\begin{align*}
c+o(1)||u_k||&=I(u_k)-\frac{1}{2 \cdot { 2^{\#}_{\mu} }(\alpha) }  \langle I'(u_k),u_k\rangle \geq  \Big(\frac{1}{2}-\frac{1}{2 \cdot { 2^{\#}_{\mu} }(\alpha) } \Big)||u_k||^2.
\end{align*}
Thus, $\{u_k\}_{k\in \N}$ is bounded in $\dot{H}^{s}(\R^{n})$, then from (\ref{eq5.8}) there exists a subsequence, still denoted by $\{u_k\}$, such that $||u_k||^2\rightarrow b$, $\int_{\R^n}{\frac{{|u_k|}^{ {2^{*}_{s}}(\beta)}}{|x|^{\beta}}}dx \rightarrow d_1 $, $B_{\alpha}(u_k,u_k) \rightarrow d_2$ and $$b=d_1+d_2.$$
By the definition of $\Lambda(n,s,\gamma,\beta)$ and $S_{\mu}(n,s,\gamma,\alpha)$, we get
\begin{equation*}
  d_1^{\frac{2}{  2^{*}_{s}(\beta) }}\Lambda(n,s,\gamma,\beta) \leq b,\quad d_2^{\frac{1}{  2^{\#}_{\mu}(\alpha) }}S_{\mu}(n,s,\gamma,\alpha)\leq b.
\end{equation*}
Therefore
\begin{equation*}
  d_1^{\frac{2}{  2^{*}_{s}(\beta) }}\Lambda(n,s,\gamma,\beta) \leq d_1+d_2,\quad d_2^{\frac{1}{  2^{\#}_{\mu}(\alpha) }}S_{\mu}(n,s,\gamma,\alpha)\leq d_1+d_2.
\end{equation*}
These inequalities lead to
\begin{equation}\label{eq5.9}
  d_1^{\frac{2}{  2^{*}_{s}(\beta) }}\Big(\Lambda(n,s,\gamma,\beta)-  d_1^{\frac{2^{*}_{s}(\beta)-2}{  2^{*}_{s}(\beta) }}\Big) \leq d_2,~~~~d_2^{\frac{1}{  2^{\#}_{\mu}(\alpha) }}\Big(S_{\mu}(n,s,\gamma,\alpha)-  d_2^{\frac{2^{\#}_{\mu}(\alpha)-1}{  2^{\#}_{\mu}(\alpha) }} \Big)\leq d_1
\end{equation}

We claim that
$$\Lambda(n,s,\gamma,\beta)-  d_1^{\frac{2^{*}_{s}(\beta)-2}{  2^{*}_{s}(\beta) }}>0,~~~~S_{\mu}(n,s,\gamma,\alpha)-  d_2^{\frac{2^{\#}_{\mu}(\alpha)-1}{  2^{\#}_{\mu}(\alpha) }}>0.$$
In fact, since $c+o(1)||u_k||=I(u_k)-\frac{1}{2}  \langle I'(u_k),u_k\rangle $, we have
\begin{align*}
\Big( \frac{1}{2}-\frac{1}{2^{*}_{s}(\beta)}\Big)
  \int_{\R^n}{\frac{{|u_k|}^{ {2^{*}_{s}}(\beta)}}{|x|^{\beta}}}dx+\Big( \frac{1}{2}-\frac{1}{2 \cdot { 2^{\#}_{\mu} }(\alpha) } \Big) B_{\alpha}(u_k,u_k)=c+o(1)||u_k||,
\end{align*}
i.e.
\begin{align} \label{eq5.010}
\Big( \frac{1}{2}-\frac{1}{2^{*}_{s}(\beta)}\Big)
 d_1+\Big( \frac{1}{2}-\frac{1}{2 \cdot { 2^{\#}_{\mu} }(\alpha) } \Big) d_2=c,
\end{align}
then
\begin{align*}
d_1  \leq  \frac{2(n-\beta)}{2s-\beta} c,~~~~~~~~
d_2  \leq  \frac{2 \cdot 2^{\#}_{\mu}(\alpha)}{2^{\#}_{\mu}(\alpha)-1} c.
\end{align*}
Using the upper bound of $d_1$, $d_2$ and the fact that $0<c<c^*$, we have
$$\Lambda(n,s,\gamma,\beta)-  d_1^{\frac{2^{*}_{s}(\beta)-2}{  2^{*}_{s}(\beta) }}\geq A_1>0,~~~~S_{\mu}(n,s,\gamma,\alpha)-  d_2^{\frac{2^{\#}_{\mu}(\alpha)-1}{  2^{\#}_{\mu}(\alpha) }}\geq A_2>0$$
where $A_1=\Lambda(n,s,\gamma,\beta)-  [\frac{2(n-\beta)}{2s-\beta} c]^{\frac{2^{*}_{s}(\beta)-2}{  2^{*}_{s}(\beta) }}$ and $A_2=S_{\mu}(n,s,\gamma,\alpha)-  [\frac{2 \cdot 2^{\#}_{\mu}(\alpha)}{2^{\#}_{\mu}(\alpha)-1} c]^{\frac{2^{\#}_{\mu}(\alpha)-1}{  2^{\#}_{\mu}(\alpha) }}$.
Thus (\ref{eq5.9}) imply
\begin{equation*}
  d_1^{\frac{2}{  2^{*}_{s}(\beta) }}A_1 \leq d_2,~~~~d_2^{\frac{1}{  2^{\#}_{\mu}(\alpha) }}A_2\leq d_1.
\end{equation*}

If $d_1=0$ and $d_2=0$, then (\ref{eq5.010}) implies that $c=0$, a contradiction with $c>0$. Therefore $d_1>0$ and $d_2>0$, we can choose $\varepsilon_0>0$ such that $d_1\geq\varepsilon_0>0$ and $d_2\geq\varepsilon_0>0$, so there exists a $K>0$ such that $k\geq K$ and
 $$
 \int_{\R^n}{\frac{{|u_k|}^{ {2^{*}_{s}}(\beta)}}{|x|^{\beta}}}dx >\varepsilon_0/2,~~~~  B_{\alpha}(u_k,u_k)>\varepsilon_0/2.$$
Then inequality (\ref{eq2.6}), the embeddings (\ref{eq1.06}) and improved Sobolev inequality (\ref{eq1.6}) imply that there exists $C>0$ such that
$$   0<C \leq ||u_k||_{  L^{2,{n-2s}+2r}(\R^{n},|y|^{-2r}) } \leq C^{-1}$$
where $r=\frac{\alpha}{ 2^*_{s}(\alpha) }$.
For any $k> K$, we may find ${\lambda}_k>0$ and $x_k \in \R^{n}$ such that
$$  {\lambda}_k^{-2s+2r} \int_{B_{{\lambda}_k}(x_k)} \frac{|u_k(y)|^2}{  |y|^{2r} }dy > ||u_k||^2_{  L^{2,{n-2s}+2r}(\R^{n},|y|^{-2r}) } -\frac{C}{2k} \geq C_1>0.$$
Let ${v}_k(x)=\lambda_k^{ \frac{n-2s}{2} }u_k({\lambda}_k x)$, since $||{v}_k||=||{u}_k||\leq C$, there exists a $v\in {\dot{H}}^s(\R^{n})$ such that
\begin{align*}
 & v_k \rightharpoonup v ~~\mbox{in} ~~\dot{H}^{s}(\R^{n})
\end{align*}
Similar to the proof of Proposition \ref{pro1.7}-(1) in Section 4, we can prove that $v \not \equiv 0$.

In addition, the boundedness of $\{{v}_k\}$ in ${\dot{H}}^s(\R^{n})$ implies that $ \{ {|v_k|}^{{2^{*}_{s}({\beta})}-2}v_k  \}$ is bounded in $L^{     \frac{{2^{*}_{s}}({\beta})}{{2^{*}_{s}({\beta})}-1}    }(\R^{n},|x|^{-{\beta}})$ and
\begin{align} \label{eq5.14}
 & {|v_k|}^{{2^{*}_{s}({\beta})}-2}v_k \rightharpoonup  {|v|}^{{2^{*}_{s}({\beta})}-2}v  ~~\mbox{in} ~~L^{     \frac{{2^{*}_{s}}({\beta})}{{2^{*}_{s}({\beta})}-1}    }(\R^{n},|x|^{-{\beta}}).
\end{align}
For any $\phi \in {L}^{2^*_{s}(\alpha)}(\R^{n},{|x|}^{-\alpha})$, Lemma \ref{lemma2.9} implies that
\begin{align} \label{eq5.15}
& \mathop {\lim }\limits_{k  \to \infty}  \int_{\R^{n}}
\big [   I_{\mu}* F_{\alpha}(\cdot,v_k)  \big](x)f_{\alpha}(x,v_k)\phi(x)dx
= \int_{\R^{n}}
\big [   I_{\mu}* F_{\alpha}(\cdot,v)  \big](x)f_{\alpha}(x,v)\phi(x)dx.
\end{align}
Since ${\dot{H}}^s(\R^{n}) \hookrightarrow  {L}^{2^*_{s}(\alpha)}(\R^{n},|x|^{-\alpha})$, then  (\ref{eq5.15}) holds for any $\phi \in {\dot{H}}^s(\R^{n})$.

Finally, we need to check that $\{v_k\}_{k\in \N}$ is also a $(PS)$ sequence for $I$ at energy level $c$. Since the norms in ${\dot{H}}^s(\R^{n})$ and $L^{{2^{*}_{s}}(\alpha)}(\R^{n},|x|^{-\alpha}) $ are invariant under the special dilation ${v}_k(x)=\lambda_k^{ \frac{n-2s}{2} }u_k({\lambda}_k x)$, we have
\begin{equation*}
 \lim_{k \to +\infty} I(v_k)=c.
\end{equation*}
Moreover, $\forall \phi \in {\dot{H}}^s(\R^{n})$, we have ${\phi}_k(x)=\lambda_k^{ \frac{2s-n}{2} }\phi(\frac{x}{{\lambda}_k})\in {\dot{H}}^s(\R^{n})$. From $I'(u_k) \to 0 ~~\mbox{in}~~\dot{H}^{s}(\R^n)'$, we can derive that
\begin{equation*}
 \lim_{k \to +\infty}\langle I'(v_k), \phi \rangle=\lim_{k \to +\infty} \langle I'(u_k), {\phi}_k\rangle=0.
\end{equation*}

Thus (\ref{eq5.14}) and (\ref{eq5.15}) lead to
\begin{equation*}
 \langle I'(v), {\phi}\rangle=\lim_{k \to +\infty} \langle I'(v_k), \phi \rangle=0.
\end{equation*}
Hence $v$ is a nontrivial weak solution of (\ref{eq1.1}). \\

\noindent\textbf{(II)} The case $s \in(0,1)$, $0\leq\alpha,\beta<2s<n$ while $\alpha \cdot \beta=0$, $\mu \in (0,n)$ and $0\leq \gamma<\gamma_{H}$.

Case (i): $\alpha=0<\beta<2s$ or $\beta=0<\alpha<2s$;

In this case, the embeddings $(\ref{eq1.06})$ and inequality (\ref{eq1.6}) are still effective. Since $\alpha>0$ or $\beta>0$, we get a nontrivial weak solution to (\ref{eq1.1}) as above by using $(\ref{eq1.06})$, (\ref{eq1.6}) and Proposition \ref{prop5.02}.

Case (ii): $\alpha=0$ and $\beta=0$;

In this case, (\ref{eq1.06}) and (\ref{eq1.6}) are useless. Since the limit equation for (\ref{eq1.1}) is
$$ (-\Delta)^{s}v={|v(x)|}^{{ 2^{*}_{s} }-2}v(x)+ \Big( \int_{\R^{n}} \frac{ {|v(y)|}^{ {2^{\#}_{\mu} }} }{ {|x-y|}^{\mu} }dy   \Big)      {|v(x)|}^{{ 2^{\#}_{\mu} }-2}v(x),      $$
by using the Nehari manifold method in \cite{JYFW}, we can also get a non-trivial weak solution to (\ref{eq1.1}) if $0\leq \gamma<\gamma_{H}$.\\
\qed

\begin{remark}
The method we adopt to prove Theorem \ref{th1.1} can be applied to prove similar existence result for the p-Laplace type problem involving double critical exponents. To go further, we consider
\begin{equation} \label{eq5.16}
 -{\Delta}_pu-{\kappa} {\frac{|u|^{p-2}u}{|x|^{p}}}=\sum_{i=1}^{2}\Big( \int_{\R^{n}} \frac{ {|u(y)|}^{ {p^{\#}_{{\mu}_i} }(\alpha_i)} }{ {|x-y|}^{{\mu}_i} {|y|}^{  {\delta_{{\mu}_i} (\alpha_i)} } }dy   \Big) \frac{      {|u(x)|}^{{ p^{\#}_{{\mu}_i} }(\alpha_i)-2}u(x)      }{  {|x|}^{  {\delta_{{\mu}_i} (\alpha_i)}            }  },
  \ \ \ x \in {\R}^{n}
 \end{equation}
where $n\geq 2$ is an integer, $p \in (1,n)$, $\kappa < \bar{\kappa}:=[(n-p)/p]^p$, ${\mu}_i \in (0,n)$, while $\alpha_i \in (0,p)$, $p^{\#}_{{\mu}_i} (\alpha_i)=(1-\frac{{\mu}_i}{2n})\cdot p^{*}(\alpha_i)$, $\delta_{{\mu}_i} (\alpha_i)=(1-\frac{{\mu}_i}{2n}){\alpha_i}$ and $ p^{*}(\alpha_i)=p(n-\alpha_i)/(n-p)$ for $i=1,2$.
We say $u \in {D}^{1,p}(\R^{n})$ is a weak solution to $(\ref{eq5.16})$ if
$$ \int_{\R^{n}}\Big[{|\nabla u|}^{p-2}\nabla u \nabla \phi-{\kappa} {\frac{|u|^{p-2}u\phi}{|x|^{p}}}\Big]=\sum_{i=1}^{2} \int_{\R^{n}} \Big( \int_{\R^{n}} \frac{ {|u(y)|}^{ {p^{\#}_{{\mu}_i} }(\alpha_i)} }{ {|x-y|}^{{\mu}_i} {|y|}^{  {\delta_{{\mu}_i} (\alpha_i)} } }dy   \Big) \frac{      {|u(x)|}^{{ p^{\#}_{{\mu}_i} }(\alpha_i)-2}u(x)\phi      }{  {|x|}^{  {\delta_{{\mu}_i} (\alpha_i)}            }  } $$
for any $\phi \in {D}^{1,p}(\R^{n})$. The following main results hold:

\begin{theorem}\label{th5.17}
The problem (\ref{eq5.16}) possesses at least a nontrivial weak solution provided either \textbf{(I)} $n\geq 2$, $p \in (1,n)$, $0<\alpha_1, \alpha_2 <p$, $0<{\mu}_1, {\mu}_2<n$ and $\kappa < \bar{\kappa}$ \\
or \textbf{(II)} $n\geq 2$, $p \in (1,n)$, $0\leq \alpha_1, \alpha_2 <p$ while $\alpha_1 \cdot \alpha_2=0$, $0<{\mu}_1, {\mu}_2<n $ and $0\leq\kappa < \bar{\kappa}$.
\end{theorem}
\end{remark}


\begin{thebibliography}{9999}

\bibitem{HBNI}
H. Br$\acute{e}$zis, L. Nirenberg, Positive solutions of nonlinear elliptic equations involving critical sobolev exponents, Comm. Pure Appl. Math. 36 (1983)437-477.


\bibitem{NGSS} N. Ghoussoub, S. Shakerian, Borderline variational problems involving fractional Laplacians and critical singularities, Adv. Nonlinear Stud. 15 (3) (2015) 527-555.


\bibitem{SDLM} S. Dipierro , L. Montoro, I. Peral, et al. Qualitative properties of positive solutions to nonlocal critical problems involving the Hardy-Leray potential, Calc. Var. Partial Differ. Equ. 55 (4) (2016) 1-29.

\bibitem{GPAP}
G. Palatucci, A. Pisante, Improved Sobolev embeddings, profile decomposition, and concentration-compactness for fractional Sobolev spaces, Calc. Var. Partial Differ. Equ. 50 (3-4) (2014) 799-829.

\bibitem{YKSS}
Y. Komori, S. Shirai. Weighted Morrey spaces and a singular integral operator, Math. Nachr. 282, No. 2, 219-231 (2009) / DOI 10.1002/mana.200610733

\bibitem{YSAW}
Y. Sawano. Generalized Morrey Spaces for Non-doubling
Measures, Nonlinear differ. equ. appl. 15 (2008), 413-425 /DOI 10.1007/s00030-008-6032-5

\bibitem{CBMO}
C. B. Morrey. On the solutions of quasi-linear elliptic partial differential equations, Trans. Amer. Math. Soc. 43, 126-166 (1938).

\bibitem{BMRW}
B. Muckenhoupt, R. Wheeden. Weighted norm inequalities for fractional integrals, Trans. Amer. Math. Soc. 192,
261-274 (1974).

\bibitem{FCZQ}
F. Catrina, Z. Q. Wang, On the Caffarelli-Kohn-Nirenberg inequalities: sharp constants,
existence (and nonexistence), and simmetry of extremal functions, Comm. Pure Appl. Math. 54
(2001) 229-258.


\bibitem{JLCS}
J. L. Chern, C. S. Lin, Minimizers of Caffarelli-Kohn-Nirenberg inequalities with the singularity
on the boundary, Arch. Ration. Mech. Anal. 197 (2010), no. 2, 401-432.

\bibitem{NGCY}
N. Ghoussoub, C. Yuan, Multiple solutions for quasi-linear PDEs involving the critical
Sobolev and Hardy exponents, Trans. Amer. Math. Soc. 12 (2000) 5703-5743.

\bibitem{NGAM}
N. Ghoussoub, A. Moradifam, Functional Inequalities: New Perspectives and New Applications,
Mathematical Surveys and Monographs, vol. 187, American Mathematical Society,
Providence, RI, 2013.

\bibitem{NGFR}
N. Ghoussoub, F. Robert, The Hardy-Schr$\ddot{o}$dinger operator with interior singularity: The remaining cases, Calc. Var. Partial Differ. Equ. 56 (5) (2016) 149.


\bibitem{Ngfr}
N. Ghoussoub, F. Robert, S. Shakerian, M. Zhao, Mass and Asymptotics associated to Fractional Hardy-Schr\"odinger Operators in Critical Regimes, Commun. Part. Differ. Equ. (2018) 1-34.

\bibitem{YLE}
Y. Lei, Asymptotic properties of positive solutions of the Hardy-Sobolev type equations, J. Differ. Equ. 254 (2013) 1774-1799.

\bibitem{GLJZ}
G. Lu, J. Zhu, Symmetry and regularity of extremals of an integral equation related to the
Hardy-Sobolev inequality, Calc. Var. Partial Differ. Equ. 42 (2011) 563-577.

\bibitem{YWYS}
Y. Wang, Y. Shen, Nonlinear biharmonic equations with Hardy potential and critical parameter, J. Math. Anal. Appl. 355 (2) (2009) 649-660.

\bibitem{DAEJ}
D. A. Lorenzo, E. Jannelli, Nonlinear critical problems for the biharmonic operator with Hardy potential, Calc. Var. Partial Differ. Equ. 54 (1) (2015) 365-396.

\bibitem{AEKS}
A. E. Khalil, S. Kellati, A. Touzani, On the principal frequency curve of the p-biharmonic operator, Arab Journal of Mathematical Sciences 17 (2) (2011) 89-99.

\bibitem{RFPP}
R. Filippucci, P. Pucci, F. Robert, On a $p$--Laplace equation with multiple critical nonlinearities, J. Math. Pures Appl. 91 (2) (2009) 156-177.

\bibitem{ENEV}
E. D. Nezza, G. Palatucci, E. Valdinoci, Hitchhikers guide to the fractional Sobolev spaces,
Bull. Sci. Math. 136 (2012) no.5 521-573.


\bibitem{LCLS}
L. Caffarelli, L. Silvestre, An extension problem related to the fractional Laplacian, Commun. Part. Differ. Equ. 32 (2007) 1245-1260.


\bibitem{WCCL}
W. Chen, C. Li, B. Ou, Classification of solutions for an integral equation, Comm. Pure Appl. Math. 59 (3) (2006) 330-343.

\bibitem{STER}
S. Terracini, On positive entire solutions to a class of equations with a singular coefficient and critical exponent, Adv. Differ. Equ. 1 (2) (1996) 241-264.



\bibitem{PLLI}
P. L. Lions, The concentration-compactness principle in the calculus of variations, The locally compact case, part 2, Ann. Inst. H. Poincar\'{e} Anal. Non.  Lin\'{e}aire 2 (1984) 223-283.

\bibitem{PLLO}
P. L. Lions, The concentration-compactness principle in the calculus of variations, The limit case,
part 1, Rev. Mat. H. Iberoamericano 1.1 (1985) 145-201.


\bibitem{EHRS}
R. L. Frank, E. H. Lieb, R. Seiringer, Hardy-Lieb-Thirring inequalities for fractional Schr$\ddot{o}$dinger operators, J. Amer. Math. Soc. 21 (4) (2008) 925-950.

\bibitem{IWHE}
I. W. Herbst, Spectral theory of the operator$(p^2+m^2)^{1/2}-Ze^2/r$, Comm. Math. Phys. 53 (1977), no. 3, 285-294.

\bibitem{ESRW}
E. Sawyer, R. L. Wheeden, Weighted inequalities for fractional integrals on Euclidean and homogeneous spaces, Am. J. Math. 114 (1992) 813-874.


\bibitem{JIAN}
J. Yang, Fractional Hardy-Sobolev inequality in $\R^N$, Nonlinear Analysis: Theory, Methods
and Applications. 119 (2015) 179-185.


\bibitem{CWJ}
W. Chen, Fractional elliptic problems with two critical Sobolev-Hardy exponents, Electronic Journal of Differential Equations 2018(2018).

\bibitem{RSER}
R. Servadei, E. Raffaella, VARIATIONAL METHODS FOR NON-LOCAL OPERATORS OF ELILIPTIC TYPE, Discrete Contin. Dyn. Syst. 33 (2013) 2105-2137.

\bibitem{RSEI}
R. L. Frank, R. Seiringer, Non-linear ground state representations and sharp Hardy inequalities, J. Funct. Anal. 255 (2008) 3407-3430.

\bibitem{YJPK}
Y. J. Park, Fractional Polya-Szeg$\ddot{o}$ inequality, J. Chungcheong Math. Soc. 24 (2011), no. 2, 267-271.

\bibitem{ELMA}
E. H. Lieb and M. Loss, Analysis, Volume 14 of Graduate Studies in Mathematics, A M S, 1997.

\bibitem{HBZI}
H. Brezis, E. Lieb, A relation between pointwise convergence of functions and convergence of functionals, Proc. Amer. Math. Soc. 88 (1983) 486-490.

\bibitem{AMPH}
A. Ambrosetti, P. H. Rabinowitz, Dual variational methods in critical point theory and applications, J. Funct. Anal. 14 (1973) 349 - 381.

\bibitem{VMJV}
V. Moroz, J. V. Schaftingen. Groundstates of nonlinear Choquard equations: Existence, qualitative properties and decay asymptotics. J. Funct. Anal, 2012, 265(2).

\bibitem{HYKD}
Y. Huang, D. Kang. On the singular elliptic systems involving multiple critical Sobolev exponents, Nonlinear Analysis, 2011, 74(2):400-412.

\bibitem{DKGL}
D. Kang, G. Li. On the elliptic problems involving multi-singular inverse square potentials and multi-critical Sobolev-Hardy exponents, Nonlinear Anal. 66 (2007), 1806-1816.

\bibitem{JYFW}
J. Yang, F. Wu. Doubly critical problems involving fractional Laplacians in ${\R}^N$. (English summary)
Adv. Nonlinear Stud. 17 (2017), no. 4, 677-690.


\bibitem{JWJP}
J. Wang, J. Shi. Standing waves for a coupled nonlinear Hartree equations with nonlocal interaction. Calc. Var. Partial Differ. Equ. 2017, 56(6):168.

\bibitem{JVIl}
J. Villavert. Qualitative properties of solutions for an integral system related to the Hardy-Sobolev inequality, J. Differ. Equ. 2015, 258(5):1685-1714.

\bibitem{ENST}
E. Nursultanov, S. Tikhonov. Weighted Norm Inequalities for Convolution and Riesz Potential, Potential Analysis, 2015, 42(2):435-456.

\bibitem{GSIN}
G. Singh, Nonlocal Pertubations of Fractional Choquard Equation, preprint, http://arxiv.org/pdf/1705.05775.

\end{thebibliography}
\end{document}